\documentclass[reqno]{amsart}
\usepackage{amsfonts}
\usepackage{mathrsfs}
\usepackage{amsmath}
\usepackage{amsthm}
\usepackage{amssymb}
\usepackage{tikz}
\usepackage{array}
\usepackage{pifont}
\usepackage{verbatim}

\newtheorem{thm}{Theorem}[section]
\newtheorem{lemma}[thm]{Lemma}
\newtheorem{cor}[thm]{Corollary}
\newtheorem{prop}[thm]{Proposition}
\newtheorem{Def}[thm]{Definition}
\newtheorem{remark}[thm]{Remark}

\newcommand{\mr}[1]{\mathscr{#1}}
\def\H{\mathcal{H}}
\def\V{\mathcal{V}}

\def\R{\mathbb{R}}
\def\N{\mathbb{N}}

\def\and{\qquad\text{and}\qquad}

\def\({\left(}
\def\){\right)}
\def\eps{\varepsilon}
\def\p{\partial}
\def\al{\alpha}

\def\Om{\Omega}

\def\wto{\rightharpoonup}

\def\la{\langle}
\def\ra{\rangle}
\def\td{{\rm d}}
\def\tdx{{\rm d}x}
\def\tdt{{\rm d}t}
\def\bin{\text{~~in~~}}

\makeatletter
\@ifundefined{c@chapter}{
\newcounter{cnum}
}{
\newcounter{cnum}[chapter]
}
\def\C{\@ifnextchar[{\@with}{\@without}}
\def\@with[#1]{\@ifundefined{c@#1}{%
      \ifnum\thecnum<1
      \stepcounter{cnum}
      \fi
      \newcounter{#1}
      \setcounter{#1}{\thecnum}
      C_{\thecnum}
      \stepcounter{cnum}
    }{
    \ifnum\the\csname c@#1\endcsname<\thecnum
     C_{{\the\csname c@#1\endcsname}}
    \else
     C_{{\the\csname c@#1\endcsname}}
     \ifnum\thecnum<1
     \stepcounter{cnum}
     \fi
    \stepcounter{cnum}
    \fi
    }}
\def\@without{\ifnum\thecnum<1
        \stepcounter{cnum}
       \fi
       C_{\thecnum}
       \stepcounter{cnum}
    }
\makeatother

\numberwithin{equation}{section}

\begin{document}
\title[wave equation with displacement dependent damping]{Well-posedness and global attractor for wave equation with displacement dependent damping and super-cubic nonlinearity}

\author{Cuncai Liu}
\email{liucc@jsut.edu.cn}
\address{School of Mathematics and Physics, Jiangsu University of Technology, Changzhou, 213001,China}

\author{Fengjuan Meng}
\email{fjmeng@jsut.edu.cn}
\address{School of Mathematics and Physics, Jiangsu University of Technology, Changzhou, 213001,China}

\author{Chang Zhang}
\email{chzhnju@126.com}
\address{School of Mathematics and Physics, Jiangsu University of Technology, Changzhou, 213001,China}
\thanks{
    This work was supported by the NSFC(12026431, 12326365, 12442050), Natural Science Fund For Colleges and
    Universities in Jiangsu Province (22KJD110001), Jiangsu 333 Project, QingLan Project of Jiangsu Province and Zhongwu Young Innovative Talents Support Program.}

\begin{abstract}
    This work investigates the semilinear wave equation featuring the displacement dependent term $\sigma(u)\partial_t u $ and nonlinearity $f(u)$. By developing refined space-time a priori estimates under extended ranges of the nonlinearity exponents with $\sigma(u)$ and $f(u)$, the well-posedness of the weak solution is established. Furthermore, the existence of a global attractor in the naturally phase space $H^1_0(\Omega)\times L^2(\Omega)$ is obtained. Moreover, the regularity of the global attractor is established, implying that it is a bounded subset of $(H^2(\Omega)\cap H^1_0(\Omega))\times H^1_0(\Omega)$.
\end{abstract}


\maketitle
\setcounter{tocdepth}{1}
\tableofcontents
\section{Introduction}
In this paper, we are concerned with the following nonlinear wave equation with displacement-dependent damping of the form
\begin{align}\label{eq1}
    \begin{cases}u_{tt}+\sigma(u)u_t-\Delta u+f(u)=\phi, ~(t,x)\in \mathbb{R}_{+}\times\Omega, \\
        u(x,0)=u_0,~~ u_{t}(x,0)=u_1,  x\in \Omega,                                   \\
        u(x,t)=0,~ (t,x)\in \mathbb{R}_{+}\times \partial\Omega.
    \end{cases}
\end{align}
where the damping coefficient $\sigma(u)$ explicitly depends on the displacement field and $\Omega$ is a bounded domain in $\mathbb{R}^3$ with smooth boundary, the external force term $\phi\in L^2(\Omega)$ is independent of time.

We impose the following assumptions on the damping coefficient $\sigma$ and the nonlinearity $f$.

\textbf{Assumptions on damping coefficient $\sigma$}. Let $\sigma\in C^1(\mathbb{R})$ and
\begin{equation}\label{S1}
    \sigma(s)\ge \sigma_0(1+|s|^r)>0,\tag{S1}
\end{equation}
and
\begin{equation}\label{S2}
    |\sigma^\prime(s)|\le cm|s|^{m-1},\quad\text{for  }|s|\ge 1,\tag{S2}
\end{equation}
where $0\le r\le m\le 4$, the extra coefficient $m$ in \eqref{S2} is for the sake of keeping compatible with the case $m=0$.

\textbf{Assumptions on nonlinearity $f$}. Let $f\in C^2(\mathbb{R})$, with $f(0)=f^{\prime}(0)=0$ and satisfies the following growth condition
\begin{equation}\label{F1}
    |f^{\prime\prime}(s)|\le  c(1+|s|^{p-2}),\quad 2\le p\le k+3,\tag{F1}
\end{equation}
where $k=\min\{r,4-m\}$, along with the dissipative condition
\begin{equation}\label{F2}
    \liminf\limits_{|s|\rightarrow\infty}\frac{f(s)}{s}> -\lambda,\tag{F2}
\end{equation}
and pseudo-monotonicity condition
\begin{equation}\label{F3}
    f^\prime(s)\ge -K,\tag{F3}
\end{equation}
where $\lambda>0$ is less than the first eigenvalue $\lambda_1$ of $-\Delta$ with Dirichlet boundary condition on $\Omega$.

Set $F(s)=\int_0^s f(\tau)\td \tau$ and $\Sigma(s)=\int_0^s\sigma(\tau)\td \tau$, under the above assumptions, we have
$$f(s)s\ge -\lambda s^2-C,\quad F(s)>-\dfrac{\lambda}{2}s^2-C,$$
and
$$ F(s)\le C(1+|s|^{p+1}), \quad |\Sigma(s)|\le C(1+|s|^{m+1}).$$

For the sake of convenience and simplicity, we denote the inner product and the norm in $L^2(\Omega)$ by $\langle\cdot,\cdot\rangle$ and $\|\cdot\|$ respectively, and the norm in $L^q(\Omega)$ by $\|\cdot\|_{q}$. Concerning the phase-spaces for our problem, we consider
$\H=H^1_0(\Omega)\times L^2(\Omega)$ and $\V=(H^2(\Omega)\cap H^1_0(\Omega))\times H^1_0(\Omega)$.

Wave phenomena arise in diverse physical, engineering, biological, and geoscientific systems. These dynamics are often modeled using various types of damping mechanisms, including: linear damping $\alpha u_t$, strong damping $-\Delta u_t$, nonlinear damping $g(u_t)$,structural damping$(-\Delta)^\alpha u_t$, displacement-dependent damping $\sigma(u)u_t$. Such models have been extensively investigated in \cite{babin92, ball04, chueshov02, chueshov04, chueshov08, feireisl95} et al.

In the one-dimensional case, Eq.\eqref{eq1} describes the dynamics of vibrating strings or analogous one-dimensional systems in media with nonlinear damping characteristics. In \cite{gatti06}, Gatti et al. proved the existence of a regular
connected global attractor of finite fractal dimension for the associated dynamical
system, as well as the existence of an exponential attractor for this  one-dimensional problem. Later, Khanmamedov \cite{khan12} extended these results by proving the existence of global attractors under weaker conditions, requiring only positivity or strict positivity of the damping coefficient function $\sigma(\cdot)$.

In the two-dimensional setting, Eq.\eqref{eq1} models wave phenomena such as vibrating membranes in stratified viscous media, playing a fundamental role in acoustics, fluid dynamics, and materials science.  Pata et al. \cite{pata07}, established the existence of a compact global attractor with optimal regularity, along with an exponential attractor. Subsequently, Khanmamedov\cite{khan10-2d} extended these results by proving the existence of global attractors under weaker conditions and further demonstrating their regularity.

In three spatial dimensions, Eq.\eqref{eq1} serves as a fundamental mathematical framework for describing nonlinear wave propagation in dissipative media. The model captures essential features of energy dissipation through its nonlinear damping terms, making it particularly relevant for several important applications including seismic wave attenuation in viscoelastic geological media, vibration damping in aerospace structural components, and stress wave propagation in nonlinear elastic solids. Pata et al. \cite{pata06} established infinite-dimensional global and exponential attractors in a
slightly weaker topology in situations where $r=0,m\leq2,p\leq3$,  Khanmamedov\cite{khan10-3d} later demonstrated the regularity of the weak attractor and the existence of a strong attractor under some additional conditions on $\sigma(\cdot)$.

Recently, Ma et al.\cite{ma} investigated the uniform attractors of Eq.\eqref{eq1} in the case of two-dimensional setting.

The asymptotic behavior of damped wave equation depends strongly on the growth rate of the nonlinearity $f$, for the linear damping, see \cite{vk16,liu17} and the references therein, for the nonlinear damping, see \cite{liu23} and the associated references.

For the displacement dependent damping wave equation \eqref{eq1}, the well-posedness and the long-term behavior also depends strongly on the growth rate $m$ of the damping coefficient $\sigma$ and $p$ of the nonlinearity $f$, moreover, $m$ strongly related to $p$. We summarize the known results through Figure 1 for clarity, the weak solution (or energy solution) of problem \eqref{eq1} is well-posed in region I, i.e., $m\le 2$ and $p\le 3$. The existence of the global attractor was investigated in region I, and also the regularity of the global attractor was proved in region I under an extra condition on damping coefficient $\sigma$.

\begin{center}
    \begin{tikzpicture}
        [cube/.style={very thick,black},
            grid/.style={very thin,gray},
            axis/.style={->,black,thick}]

        \draw[axis] (0,0) -- (6,0) node[anchor=west]{$p$};
        \draw[axis] (0,0) -- (0,5) node[anchor=west]{$m, r$};
        \coordinate (A) at (1,1);

        \draw (0,0) -- (0, 0.1);
        \node[below] at (0,0){0};

        \draw (1,0) -- (1, 0.1);
        \node[below] at (1,0){1};

        \draw (3,0) -- (3, 0.1);
        \node[below] at (3,0){3};

        \draw (5,0) -- (5, 0.1);
        \node[below] at (5,0){5};

        \draw (0,2) -- (0.1,2);
        \node[left] at (0,2){2};

        \draw (0,4) -- (0.1, 4);
        \node[left] at (0,4){4};

        \draw[thick, fill=black!10] (1,0) -- (1,2) -- (3,2) -- (3,0) -- (1,0);
        \draw[thick, fill=black!20]  (1,2) -- (1,4) -- (3,4)-- (5,2) -- (3,0) -- (3,2) -- (1,2);

        \draw[thick, dashed] (5,0) -- (5,4) -- (3,4);
        \draw[thick] (1,4) -- (3,4);

        \draw[thick] (3,0) -- (5,2) -- (3,4);

        \node[font=\fontsize{8}{8}] at (2,1){I};
        \node[font=\fontsize{8}{8}] at (3,3){II};
        \node[font=\fontsize{8}{8}] at (4.3,3.5){$p \le 7-m$};
        \node[font=\fontsize{8}{8}] at (3.8,1.5) {$p\le r+3$};

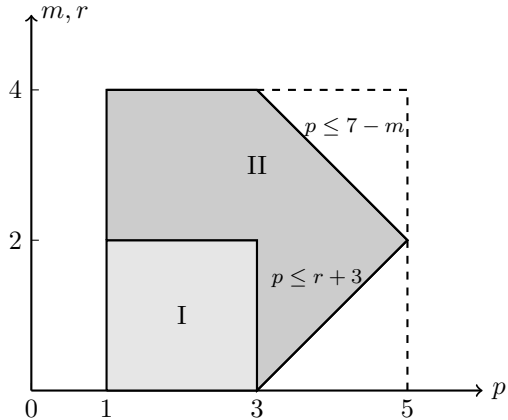
        \node at (3,-0.8) {Figure 1. Region of exponents for damping term and nonlinearity};
    \end{tikzpicture}
\end{center}

In \cite{liu23}, the authors investigate the wave equation with nonlinear interior damping
$$u_{tt}+g(u_t)-\Delta u+f(u)=\phi,$$
while the damping function $g(s)\approx  |s|^{m-1}s$. 
Utilizing the dissipativity relation 
$$\int_{0}^{+\infty}\int_{\Omega}g(u_{t})u_{t}\td x\td t\le +\infty,$$
and a nonlinear test function, we can establish a priori space-time estimates for the weak solution. These estimates enable us to handle the nonlinear term $f(u)$ with growth exponents up to $\min\{3m,5\}$.

Notice that in the definition of weak solution (Definition\ref{def_weak}) of equation \eqref{eq1}, the dissipative mechanism encoded in $\sigma(u)u_t$ yields $\int_0^\infty\int_\Omega \sigma(u)|u_t|^2 \tdx\tdt<\infty$, which fundamentally transcends the energy-space containment. Following the strategy employed for interior damping case, we can also construct a priori space-time estimates for Equation \eqref{eq1} via appropriately designed nonlinear test functions. Then we can extend exponents for both the damping term and nonlinearity to the parameter space defined by region II as illustrated in Figure 1.


The paper is organized as follows. Section 2 is devoted to the a priori estimate of weak solution. In Section 3, well-posedness of the weak solution is establised. The existence of the global attractor is the focus of Section 4. In the end, the regularity of the global attractor in $(H^2(\Omega)\cap H^1_0(\Omega))\times H^1_0(\Omega)$ is shown in Section 5.

Throughout the paper, the symbol $C$ stands for a generic constant indexed occasionally for the sake of clarity, the different positive constants $C_i, i\in \mathbb{N}$ also used for special differentiation in the paper. In the following discussion, $ a\lesssim b$ means $ a\le  Cb$ for various values of $C$.

\section{The a priori estimate of weak solution }
In this section, we dedicate to the a priori estimate of weak solution for the problem \eqref{eq1}. The estimate not only guarantee the uniqueness of the weak solution and the validity of the energy equality, but also serve to prove the norm-to-weak continuity of the associated semigroup. Fisrtly, let's recall the definition of weak solution for problem \eqref{eq1}.
\begin{Def}[Weak solution]\label{def_weak}
    A function $u$ is a weak solution of problem \eqref{eq1} if for any $T>0$
    $$\xi_u(t)=(u(t),u_t(t))\in L^\infty(0,T;\H), \quad \sigma(u)|u_t|^2\in L^1([0,T]\times\Omega)$$
    and equation \eqref{eq1} is satisfied in the sense of distribution, i.e.
    $$\int_{0}^{T}\int_\Omega \left(-u_t\psi_t+\sigma(u)u_t\psi+\nabla u\nabla \psi+f(u)\psi\right)\td x \td t=\int_{0}^{T}\int_\Omega \phi\psi\td x\td t$$
    for any $\psi \in C_0^{\infty}((0,T)\times \Omega)$.
\end{Def}

\begin{thm}[The a priori estimate of weak solution]\label{priori}
    Assume that $\phi\in L^2(\Omega)$ and conditions \eqref{S1}-\eqref{F2} hold. Suppose that $u(t)$ is a weak solution of equation \eqref{eq1} on $[0, T]$, then the following space-time estimate
    \begin{equation}\label{space-time}
        \| u\|_{L^{k+2}(0,T;L^{3k+6}(\Omega))}\le C(R^4T+R^5)
    \end{equation}
    holds, where $R=\sup_{0\le t\le T}(\|u(t)\|_6+\|u_t(t)\|)+\int_0^T\int_\Omega \sigma(u)u_t^2\td x\td t$+1 and the constant $C$ is independent of $u$ and $T$.
\end{thm}
\begin{proof}
    For $\eps\in (0,1)$, we choose a nonlinear test function $M_\eps(s)=\dfrac{|s|^{k}}{1+\eps|s|^{k}}s$, then $M_\eps^\prime(s)=\dfrac{k+1+\eps|s|^{k}}{(1+\eps|s|^{k})^2}|s|^{k}$. We have the following estimates
    \begin{equation}
        |M_\eps(s)|\le \frac{|s|}{\eps},\quad 0\le M_\eps^\prime(s)\le \frac{k+1}{\eps}
    \end{equation}
    and
    \begin{equation}\label{H_uni}
        |M_\eps(s)|\le |s|^{k+1},\quad 0\le M_\eps^\prime(s)\le (k+1)|s|^{k}.
    \end{equation}
    Therefore $M_\eps(u)\in L^\infty(0,T;H^1_0(\Omega))$, this fact allows us to multiply the equation \eqref{eq1} by $M_\eps(u)$, then integrate on $[0, T]\times \Omega$, the resulting identity is
    \begin{equation}\label{eqh}
        \int_{0}^T\int_\Omega\left(M_\eps(u)u_{tt}+M_\eps^{\prime}(u)|\nabla u|^2+M_\eps(u)(\sigma(u)u_t+f(u)-\phi)\right)\td x\td t=0.
    \end{equation}

    The space-time norm $\|\cdot\|_{L^{k+2}(0,T;L^{3k+6}(\Omega))}$ derives from the second term on the left hand side. We observe that
    \begin{equation}
        M_\eps^\prime(u)|\nabla u|^2\ge \frac{|u|^{k}}{1+\eps|u|^{k}}|\nabla u|^2=|\nabla m_\eps(u)|^2
    \end{equation}
    where$$ m_\eps(u)=\int_0^{|u|}\frac{s^{\frac{k}2}}{({1+\eps s^{k}})^{\frac12}}\td s.$$
    Then by Sobolev embedding inequality, we have
    \begin{equation}\label{hd}
        \int_\Omega M_\eps^\prime(u)|\nabla u|^2\td x\ge \|\nabla m_\eps(u)\|^2\ge C_1 \|m_\eps(u)\|^{2}_{6}.
    \end{equation}
    Using integration by parts on $t$, we can obtain
    \begin{equation}\label{ibp}
        \int_{0}^T\int_\Omega M_\eps(u)u_{tt} \td x\td t= \left.\int_\Omega M_\eps(u(t))u_{t}(t)\td x\right|_0^T-\int_{0}^T\int_\Omega M_\eps^\prime(u)|u_{t}|^2 \td x\td t
    \end{equation}
    and
    \begin{equation}
        \int_\Omega M_\eps^{\prime}(u)|u_t|^2\td x\le \int_\Omega (k+1)|u|^{k}|u_t|^2\td x\le C \int_\Omega \sigma(u)|u_t|^2 \td x.
    \end{equation}
    Since $k\le 2$, we can obtain that
    \begin{equation}\label{hut}
        \left\vert\int_\Omega M_\eps(u)u_t\td x\right\vert \lesssim \|u\|_{6}^{k}\|u_t\|
    \end{equation}
    and
    \begin{equation}\label{hphi}
        \left|\int_\Omega M_\eps(u)\phi \td x\right| \lesssim \|u\|_{6}^{k}\|\phi\|.
    \end{equation}
    For the damping term in \eqref{eqh}, according to \eqref{S2} and the fact $ |M_\eps(s)|\le  C|s|^{\frac{k}{2}}m_\eps(s)$, we have
    \begin{equation}
        \begin{aligned}
            \left\vert\int_\Omega M_\eps(u)\sigma(u)u_t \td x\right\vert & \le  \left(\int_\Omega M_\eps^2(u)\sigma(u) \td x\int_\Omega \sigma(u)|u_t|^2 \td x\right)^{\frac12}      \\
                                                                         & \le  C\left(\int_\Omega (1+|u|^{k+m})m_\eps^2(u) \td x\int_\Omega \sigma(u)|u_t|^2 \td x\right)^{\frac12} \\
                                                                         & \le  C\|m_\eps(u)\|_6 \left((1+\|u\|_6^{k+m})\int_\Omega \sigma(u)|u_t|^2 \td x\right)^{\frac12}          \\
                                                                         & \le \frac{C_1}2\|m_\eps(u)\|^2_6+C(1+\|u\|_6^{k+m}) \int_\Omega \sigma(u)|u_t|^2 \td x.
        \end{aligned}
    \end{equation}
    Applying condition \eqref{F2}, we obtain
    \begin{equation}\label{hf}
        \int_\Omega M_\eps(u)f(u)\td x\ge -\lambda_1 \int_\Omega |u|^{k+1}\td x-C\ge -C\|u\|_6^{k+1}-C.
    \end{equation}
    Substituting \eqref{hd}-\eqref{hf} into identity \eqref{eqh}, we deduce
    \begin{equation}
        \int_0^T \|m_\eps(u)\|^2_6\td t \le C(R^4T+R^5).
    \end{equation}
    Let $\eps\to 0^{+}$, since $m_\eps(u)\to |u|^{\frac{k+2}2}$ a.e. in $[0,T]\times \Omega$, then by Fatou's lemma, we can obtain
    \begin{equation}\label{ste}
        \int_0^T \|u\|_{3k+6}^{k+2}\td t \le C(R^4T+R^5).
    \end{equation}
\end{proof}
\begin{remark}\label{rem_fu}
    By \eqref{S1} and H\"{o}lder's inequality, we have
    \begin{equation}
        \int_\Omega f^2(u)\sigma^{-1}(u)\td x\le C \int_\Omega (1+|u|^{2p-r})\td x\le C+C\|u\|^{\alpha (2p-r)}_{3k+6}\|u\|^{(1-\alpha) (2p-r)}_{6},
    \end{equation}
    where
    $$\alpha=(\frac16-\frac{1}{2p-r})/(\frac16-\frac1{3k+6}).$$
    Since
    $$\alpha(2p-r)=\frac{(2p-r-6)(k+2)}{k}\le k+2,$$
    by space-time estimate \eqref{space-time}, for any weak solution $u(t)$ of problem \eqref{eq1}, we can deduce that
    $$f(u)\sigma^{-\frac12}(u)\in L^{2}((0,T)\times\Omega).$$
\end{remark}

\section{Well-posedness of weak solution}
This section is devoted to investigating the well-posedness of weak solutions for problem \eqref{eq1}. We initiate our analysis by establishing existence through the Faedo-Galerkin approximation scheme. Building upon this existence result, we subsequently derive an energy equality via variational techniques. This energy identity plays a pivotal role in demonstrating the asymptotic compactness of the associated semigroup through rigorous energy dissipation estimates. Furthermore, we prove the norm-to-weak continuity of the semigroup operator by employing weak convergence arguments and compactness criteria in Sobolev spaces.

\subsection{Existence of weak solutoin}
\begin{thm}\label{weak}
    Suppose that all assumptions of Theorem \ref{priori} hold. Then, for every $(u_0,u_1) \in \H$, there exists a global solution $u(\cdot)$ of equation \eqref{eq1} and the following dissipative estimate holds
    \begin{equation}\label{weak_energy_est}
        \| (u(t),u_t(t))\|_{\H} \le Q(\| (u_0,u_1)\|_{\H})e^{-t}+Q(\|\phi\|),
    \end{equation}
    $Q$ is a continuous increasing function.
\end{thm}
\begin{proof}
    The existence part will be proved via the Faedo-Galerkin method. We know that there exists a complete orthonormal basis of $L^2(\Omega)$, $\{e_i\}_{i\in \mathbb{N}}$ made of eigenvectors of $-\Delta$ with Dirichlet boundary condition,
    $$-\Delta e_i=\lambda_ie_i,~~ i\in \mathbb{N},$$
    where the eigenvalues $\lambda_1$, $\lambda_2$, $\cdots$ are increasing and tend to $+\infty$. It's clear that $e_i$ is smooth enough since $\Omega$ is smooth.

    Let $P_n$ be the orthogonal projection in $L^2(\Omega)$ onto the span of $\{e_1,e_2,...,e_n\}$. Considering the following finite dimensional approximate problem
    \begin{equation}\label{eq_fdap}
        \begin{cases}
            u^n_{tt}+P_n (\sigma(u^n)u^n_t)-\Delta u^n+P_nf(u^n)=\phi_n=P_n \phi, \\
            u^n(0)=P_nu_0,u^n_t(0)=P_nu_1,
        \end{cases}
    \end{equation}
    which is actually an ordinary differential system.

    Since $\sigma\in C^1(\R)$ and $f\in C^2(\R)$, by Picard's Existence Theorem, this system has a local classical solution
    $$u^n(t)=\sum_{i=1}^{n}a_{in}(t)e_i.$$
    Multiplying \eqref{eq_fdap} by $u^n_t(t)$ and integrate over $\Omega$, we get
    \begin{equation}\label{eq2.7}
        \frac{\td}{\td t}E(u^n(t))+\int_\Omega \sigma(u^n) (u^n_t)^2\td x=0,
    \end{equation}
    where the energy functional
    $$E(u^n(t))=\frac12\|u^n_t\|^2+\frac12\|\nabla u^n\|^2+\int_\Omega F(u^n)\td x-\int_\Omega \phi u^n\td x.$$
    Integrating \eqref{eq2.7} from $0$ to $t$, we obtain
    \begin{equation}\label{eq2.8}
        E(u^n(t))+\int_0^t \int_\Omega \sigma(u^n) (u^n_t)^2\td x\td s=E(u^n(0)).
    \end{equation}
    The growth condition \eqref{F1} and dissipative condition \eqref{F2} implies that
    $$C_1(\|u^n_t(t)\|^2+\|\nabla u^n(t)\|^2-C_2)\le  E(u^n(t))\le C_2\left(\|u^n_t(t)\|^2+\|\nabla u^n(t)\|^2+1\right)^{p/2}.$$
    Then we have the energy estimate
    \begin{equation}\label{eq2.9}
        \|u^n_t(t)\|^2+\|\nabla u^n(t)\|^2+\int_0^t\int_\Omega \sigma(u^n) (u^n_t)^2\td x \td s \le C_3(E(u^n(0))+1),
    \end{equation}
    which implies that the local solution can be extended to any finite time interval.

    By virtue of \eqref{S2}, we have $\sigma^{1/2}(u^n) \in L^{\infty}(0,T;L^{3}(\Omega))$.
    Applying H\"older inequality, we obtain that $\sigma(u^n)u^n_t\in L^{6/5}([0,T]\times\Omega)$.

    Without loss of generality, we can assume that
    \begin{align*}
         & u^n\to u \text{ weakly star in } L^\infty(0,T;H^1_0(\Omega)),      \\
         & u^n_t\to u_t \text{ weakly start in } L^{\infty}(0,T;L^2(\Omega)), \\
         & u^n\to u \text{ in } L^5((0,T)\times\Omega).                       
    \end{align*}
    According to growth condition \eqref{S2} and \eqref{F1}, we have
    $$\Sigma(u^n)\to \Sigma(u),~~f(u^n)\to f(u) \text{~~in~~} L^{1}((0,T)\times\Omega),$$
    and
    $$\Sigma(u^n)\wto \Sigma(u),~~~f(u^n)\wto f(u) \text{~~in~~} L^{6/5}((0,T)\times\Omega).$$
    Since $\sigma(u^n)u^n_t=\dfrac{\td }{\td t}\Sigma(u^n)$ and $\sigma(u)u_t=\dfrac{\td }{\td t}\Sigma(u)$ in the sense of distribution, then $\sigma(u^n)u^n_t \to \sigma(u)u_t$ weak in $L^{6/5}((0,T)\times\Omega)$.

    Passing to the limit as $n\to \infty$ in equation \eqref{eq_fdap}, we can obtain that the limit function $u$ is a global weak solution of equation \eqref{eq1}.

    The dissipative estimate \eqref{weak_energy_est} can be deduced from multiplying \eqref{eq_fdap} by $u^n+\al u^n_t$ for some sufficiently small constant $\al>0$, which is a standard approach( see \cite{pata06} for instance), so we omit the proof here.
\end{proof}

\subsection{Energy equality}
As mentioned in Remark \eqref{rem_fu}, the a priori space-time estimate \eqref{space-time} implies that $f(u)\sigma^{-\frac12}(u)\in L^{2}((0,T)\times\Omega)$. Then taking into account $\sigma^{\frac12}(u)u_t\in L^{2}((0,T)\times\Omega)$, the product
$$f(u)u_t=f(u)\sigma^{-\frac12}(u) \sigma^{\frac12}(u)u_t$$
belongs to $L^{1}((0,T)\times\Omega)$, which means that multiplying equation \eqref{eq1} by $u_t$ will make sence.
\begin{thm}
    Suppose that all assumptions of Theorem \ref{priori} hold. Then for any weak solution $u(t)$ of problem \eqref{eq1}, the energy functional
    $$E(u(t))=\frac12\|u_t(t)\|^2+\frac12\|\nabla u(t)\|^2+\int_\Omega (F(u)-\phi u)\td x$$
    is absolutely continuous on $\R^{+}$ and the following energy equality
    \begin{equation}
        \frac{\td}{\td t}E(u(t))+\int_\Omega \sigma(u)u_t^2\td x=0
    \end{equation}
    holds almost everywhere on $\R^{+}$.
\end{thm}
\begin{proof}
    Since $\sigma(s)\ge \sigma_0>0$, thus for any weak solution $u(t)$ of problem \eqref{eq1}, we have $\sigma^{-\frac12}(u)\in L^{\infty}([0,T]\times\Omega)$.

    On the other hand, we can rewrite equation \eqref{eq1} as
    \begin{equation}\label{re_eq}
        u_{tt}-\Delta u=\phi-f(u)-\sigma(u)u_t.
    \end{equation}
    Since $u\in L^{\infty}(0,T;H^1_0(\Omega))$, by virtue of \eqref{S2} and \eqref{F1}, we have
    $$\sigma^{\frac12}(u) \in L^{\infty}(0,T;L^{3}(\Omega)) \and f(u)\in L^{\infty}(0,T;L^{\frac65}(\Omega)).$$
    Taking account into $\sigma^{1/2}(u)u_t \in L^2([0,T]\times\Omega)$, we obtain that $\sigma(u)u_t\in L^{6/5}([0,T]\times\Omega)$. It concluded that the right side terms of equality \eqref{re_eq} belong to $L^{6/5}([0,T]\times\Omega)$, which yields that we can multiply both sides of equality \eqref{re_eq} by $\sigma^{-\frac12}(u)$ to obtain that
    \begin{equation}\label{eq_sigma}
        (u_{tt}-\Delta u)\sigma^{-\frac12}(u)=\phi\sigma^{-\frac12}(u)-f(u)\sigma^{-\frac12}(u)-\sigma^{\frac12}(u)u_t.
    \end{equation}
    Recall that from Remark \ref{rem_fu} that $f(u)\sigma^{-\frac12}(u)\in L^{2}([0,T]\times\Omega)$, thus all three terms on the right side of equality \eqref{eq_sigma} belong to $L^{2}([0,T]\times\Omega)$. This fact allows us to take the inner product of equality \eqref{eq_sigma} with $\sigma^{1/2}(u)u_t$, which derive that
    \begin{equation}\label{eq_ut}
        \int_\Omega (u_{tt}-\Delta u)u_t\td x=\int_\Omega (\phi u_t-f(u)u_t-\sigma(u)u_t^2)\td x.
    \end{equation}
    Approximating the function $u$ by smooth ones and arguing in a standard way (see \cite[Ch2, Lemma 4.1]{temam97} for instance), we conclude from the left hand side term that
    \begin{equation}\label{energy_left}
        \int_\Omega (u_{tt}-\Delta u)u_t\td x=\frac{\td}{\td t}\left(\frac12\|u_t\|^2+\frac12\|\nabla u\|^2\right).
    \end{equation}
    And for the right hand side terms, we have
    \begin{equation}\label{energy_right}
        \int_\Omega (\phi u_t-f(u)u_t-\sigma(u)u_t^2)\td x=\frac{\td}{\td t}\left(\int_\Omega (\phi u-F(u))\td x\right)-\int_\Omega \sigma(u)u_t^2\td x.
    \end{equation}
    Substituting above two equalities into equality \eqref{eq_ut}, we have the energy equality
    \begin{equation}
        \frac{\td}{\td t}E(u(t))+\int_\Omega \sigma(u)u_t^2\td x=0,
    \end{equation}
    holds a.e. on $\mathbb{R}^{+}$.
\end{proof}

\subsection{Norm-to-weak continuity}
To complete the proof of the well-posedness, it remains to prove some kind of continuous dependence of weak solution w.r.t the initial data. Here, we involve the so-called norm-to-weak continuity of the weak solution.
\begin{Def}[\cite{zhong}]
    Let $X$ be a Banach space and $S(t)$ be an operator semigroup on $X$. We say $S(t)$ is a norm-to-weak continuous on $X$, if $S(t_n)x_n\wto S(t)x$ where $t_n\to t$ and $x_n\to x$ in $X$.
\end{Def}
Firstly, we will establish the $L^2$-norm estimate on the difference $u^1(t)-u^2(t)$ for any two weak solutions $u^1(t)$, $u^2(t)$.
\begin{prop}\label{partial_cont}
    Let $u^1(t)$ and $u^2(t)$ be two weak solutions of \eqref{eq1} in $[0, T]\times\Omega$. Then we have the following estimate holds
    \begin{equation}\label{ieq_cd}
        \|\bar{u}(t)\|\le CE^{CT}\|\xi_{\bar{u}}(0)\|_{\H}
    \end{equation}
    where $\bar{u}(t)=u^1(t)-u^2(t)$ and $C$ may depends on $\|\xi_{u^1}(t)\|_{\H}$ and $\|\xi_{u^2}(t)\|_{\H}$.
\end{prop}
\begin{proof}
    For any weak solutions $u^1(t)$ and $u^2(t)$, denote $\bar{u}=u^1-u^2$ and
    $$R=\max_{[0, T]}(\|\xi_{u^1}(t)\|_{\H}+\|\xi_{u^2}(t)\|_{\H})+1.$$

    Defining $w^i(t)= \int_0^t u^i(s)\td s$, $i=1, 2$. Denote $\bar{w}=w^1-w^2$ and $\xi_{\bar{w}}(t)=(\bar{w}(t), \bar{w}_t(t))=(\bar{w}(t), \bar{u}(t))$. Integrating equation \eqref{eq1} for two weak solutions on $[0, t]$ and taking the difference yields
    \begin{equation}\label{barw}
        \partial_t^2 \bar{w}+\Sigma(u^1)-\Sigma(u^2)-\Delta \bar{w}=F+G,
    \end{equation}
    where
    \begin{equation*}
        F(t)=\int_0^t[f(u^2(\tau))-f(u^1(\tau))]\td \tau \and G=\Sigma(u^1(0))-\Sigma(u^2(0))+\partial_t \bar{u}(0).
    \end{equation*}
    Note that, on account of \eqref{S2} and \eqref{F1}, all the terms of \eqref{barw} belong at least to $L^2(0, T;H^{-1})$. Hence, their product with $\partial_t\bar{w}=\bar{u}\in L^{\infty}(0,T;H^1_0(\Omega))$ is well defined. Taking this product, and observing
    that $\la\Sigma(u^1)-\Sigma(u^2),\bar{u} \ra\ge 0$, we get
    \begin{equation}
        \frac{\td}{\td t}\|\xi_{\bar{w}}(t)\|_{\H}^2\le 2\frac{\td}{\td t}\la F,\bar{w}\ra+2\frac{\td}{\td t}\la G,\bar{w}\ra-2\la\partial_tF,\bar{w}\ra.
    \end{equation}
    Integrating on $[0, t]$, we are led to
    \begin{equation}\label{ebw}
        \begin{aligned}
            \|\xi_{\bar{w}}(t)\|_{\H}^2 & \le \|\xi_{\bar{w}}(0)\|_{\H}^2+2\la F(t),\bar{w}(t)\ra+2\la G,\bar{w}(t)\ra-2\int_{0}^{t}\la\partial_{\tau}F(\tau),\bar{w}(\tau)\ra\td \tau \\
                                        & \le \|\bar{u}(0)\|^{2}+2\|F(t)\|_{H^{-1}}\|\xi_{\bar{w}}(t)\|_{\H}+2\|G\|_{H^{-1}}\|\xi_{\bar{w}}(t)\|_{\H}                                  \\
                                        & +2\int_{0}^{t}\|\partial_{\tau}F(\tau)\|_{H_{-1}}\|\xi_{\bar{w}}(\tau)\|_{\H} \td \tau.
        \end{aligned}
    \end{equation}
    Using now the growth restrictions \eqref{S2} and \eqref{F1} on $\sigma$ and $f$, we easily see that
    \begin{equation}\label{g-1}
        2\|G\|_{H^{-1}} \le C(R)\|\xi_{\bar{u}}(0)\|_{\H},
    \end{equation}
    and
    \begin{equation}
        \begin{aligned}
            2\|F(t)\|_{H^{-1}} & \le 2\int_0^t\|f(u^2(\tau))-f(u^1(\tau))\|_{H^{-1}}\td \tau                                                 \\
                               & \le C\int_0^t\|f(u^2(\tau))-f(u^1(\tau))\|_{L^{6/5}(\Omega)}\td \tau                                        \\
                               & \le C\int_0^t(1+\|u^1(\tau)\|_{3p-3}^{p-1}+\|u^2(\tau)\|_{3p-3}^{p-1})\|\xi_{\bar{w}}(\tau)\|_{\H}\td \tau.
        \end{aligned}
    \end{equation}
    Since $p\le k+3$,
    \begin{equation}
        2\|F(t)\|_{H^{-1}} \le \int_0^t C_1(1+\|u^1(\tau)\|_{3k+6}^{k+2}+\|u^2(\tau)\|_{3k+6}^{k+2})\|\xi_{\bar{w}}(\tau)\|_{\H}\td \tau.
    \end{equation}
    Denote $H(\tau)=C_1(1+\|u^1(\tau)\|_{3k+6}^{k+2}+\|u^2(\tau)\|_{3k+6}^{k+2})$, then
    \begin{equation}
        2\|F(t)\|_{H^{-1}}\le \int_0^t H(\tau)\|\xi_{\bar{w}}(\tau)\|_{\H}\td \tau.
    \end{equation}
    By an analogy analysis,
    \begin{equation}\label{pF-1}
        2\|\partial_{\tau}F(\tau)\|_{H^{-1}} =2 \|f(u^2(\tau))-f(u^1(\tau))\|_{H^{-1}} \le H(\tau)\|\xi_{\bar{w}}(\tau)\|_{\H}.
    \end{equation}
    Substituting \eqref{g-1}-\eqref{pF-1} into \eqref{ebw}, we have
    \begin{equation}
        \begin{aligned}
            \|\xi_{\bar{w}}(t)\|_{\H}^2 & \le \|\xi_{\bar{w}}(0)\|_{\H}^2+\int_0^t H(\tau)\|\xi_{\bar{w}}(\tau)\|_{\H}\td \tau\|\xi_{\bar{w}}(t)\|_{\H}   \\
                                        & +C(R)\|\xi_{\bar{u}}(0)\|_{\H}\|\xi_{\bar{w}}(t)\|_{\H}+\int_0^t H(\tau)\|\xi_{\bar{w}}(\tau)\|_{\H}^2\td \tau.
        \end{aligned}
    \end{equation}
    Denote $\Phi(t)=\sup_{0  \le \tau\le t}\|\xi_{\bar{w}}(\tau)\|_{\H}$ and take the supermum w.r.t. time on the both sides of above inequality, we have
    \begin{equation}
        \Phi(t)^2  \le \|\xi_{\bar{w}}(0)\|_{\H}\Phi(t)+2\int_0^t H(\tau)\|\xi_{\bar{w}}(\tau)\|_{\H}\td \tau\Phi(t)+C(R)\|\xi_{\bar{u}}(0)\|_{\H}\Phi(t).
    \end{equation}
    Together with $\|\xi_{\bar{w}}(0)\|_{\H}\le C \|\xi_{\bar{u}}(0)\|_{\H}$ yields
    \begin{equation}
        \|\xi_{\bar{w}}(t)\|_{\H}  \le \Phi(t) \le +2\int_0^t H(\tau)\|\xi_{\bar{w}}(\tau)\|_{\H}\td \tau +C(R)\|\xi_{\bar{u}}(0)\|_{\H}.
    \end{equation}
    From the Gronwall lemma we conclude that
    \begin{equation}
        2\int_0^t H(\tau)\|\xi_{\bar{w}}(\tau)\|_{\H}\td \tau\le Q(R)\left(e^{2\int_0^t H(\tau)\td \tau}-1\right)\|\xi_{\bar{u}}(0)\|_{\H}.
    \end{equation}
    Therefore
    \begin{equation}
        \|\xi_{\bar{w}}(t)\|_{\H}\le C(R) e^{2\int_0^t H(\tau)\td \tau}\|\xi_{\bar{u}}(0)\|_{\H}.
    \end{equation}
    Taking account into \eqref{space-time}, we have
    \begin{equation}
        \|\xi_{\bar{w}}(t)\|_{\H}\le C(R,T)\|\xi_{\bar{u}}(0)\|_{\H}.
    \end{equation}
\end{proof}
As a conclusion, we can obtain the uniqueness of weak solution.
\begin{cor}[Uniqueness]
    Furthermore, if $\|\xi_{\bar{u}}(0)\|_{\H}=0$, then the two solutions is actually identity, which means that the weak solution is unique.
\end{cor}
We are now ready to define the solution semigroup $S(t):\H\rightarrow\H $ associated with problem \eqref{eq1}:
\begin{eqnarray}
    S(t):\xi_u(0)\rightarrow \xi_u(t)
\end{eqnarray}
where $u(t)$ is the unique weak solution of \eqref{eq1}.

Now, let's prove norm-to-weak continuity of the solution semigroup $S(t)$ on $\H$.
\begin{thm}[Norm-to-weak continuity]\label{ntw_cont}
    The semigroup $S(t)$ is norm-to-weak continuous on $\H$ for any $t>0$.
\end{thm}
\begin{proof}
    Suppose that sequences $t_n\to t^{*}\in \R^{+}$ and $z_n=(u^n_0,u^n_1)$ in $\H$ such that $z_n\to z$ in $\H$ for some $z=(u_0,u_1)\in \H$. According to the energy estimate \eqref{weak_energy_est}, we have
    $$R:=\sup_n\sup_{t\in\R^{+}}\|\xi_{u^n}(t)\|_{\H}+1<+\infty.$$
    Since $\H$ is dense in $\H^{-1}$, after a simliar approximate argument with \cite[Theorem 3.2]{zhong}, we just need to prove that $\xi_{u^n}(t^n)\wto \xi_{u}(t^{*})$ in $\H^{-1}$.
    We will devide the proof into two steps. Firstly, let prove that $\xi_{u^n}(t)$ is local H\"older continuous on time in $\H^{-1}$--norm uniformly w.r.t. $n\in \N$.
    By the fundamental theorem of calculus, we have the following Lipschitz continuity
    \begin{equation}\label{u_cont}
        \|u^n(t_n)-u^n(t)\|=\left\|\int_t^{t_n} u_t^n(\tau)\td \tau\right\|\le \left|\int_t^{t_n}\| u_t^n(\tau) \|\td\tau\right|\le  R|t^n-t|.
    \end{equation}
    For the derivative $u^n_t$, we find
    \begin{equation}
        \|u^n_t(t_n)-u^n_t(t)\|_{H^{-1}}=\left\|\int_t^{t_n} u_{tt}^n(\tau)\td \tau\right\|_{H^{-1}}\le \left|\int_t^{t_n}\| u_{tt}^n(\tau) \|_{H^{-1}}\td\tau\right|.
    \end{equation}
    To obtain a similar continuity, we need to estimate the bound of second order derivative $u^n_{tt}$ in $H^{-1}$. By equation \eqref{eq1},
    \begin{equation}
        u_{tt}^n(t)=\phi+\Delta u^n(t)-\sigma(u^n(t))u^n_t(t)-f(u^n(t)),
    \end{equation}
    then
    \begin{equation}\label{utt-1}
        \|u_{tt}^n(t)\|_{H^{-1}} \le C\|\phi\|+\|\nabla u^n(t)\|+\|\sigma(u^n(t))u^n_t(t)\|_{H^{-1}}+\|f(u^n(t))\|_{H^{-1}}.
    \end{equation}
    With regard to the damping term on the right hand side, H\"older inequality implies that
    \begin{equation}
        \begin{aligned}
            \|\sigma(u^n(t))u^n_t(t)\|_{H^{-1}} & \le C\|\sigma(u^n(t))u^n_t(t)\|_{\frac65}                           \\
                                                & \le C\|\sqrt{\sigma(u^n(t))}\|_{3}\|\sqrt{\sigma(u^n(t))}u^n_t(t)\| \\
                                                & \le C(1+\|u^n(t)\|_{6}^{\frac m2})\|\sqrt{\sigma(u^n(t))}u^n_t(t)\| \\
                                                & \le CR^{\frac{m}{2}}\|\sqrt{\sigma(u^n(t))}u^n_t(t)\|.
        \end{aligned}
    \end{equation}
    Taking account into the dissipation integrals,
    \begin{equation}\label{su}
        \int_t^{t_n}  \|\sigma(u^n(t))u^n_t(t)\|_{H^{-1}}\td\tau \le  \int_t^{t_n} CR^{\frac m2}\|\sqrt{\sigma(u^n(\tau))}u^n_t(t)\|\td\tau \le CR^{\frac{m+1}{2}}|t^n-t|^{\frac12}.
    \end{equation}
    Therefore, substituting into \eqref{utt-1}, we can obtain
    \begin{equation}\label{ut_cont}
        \|u^n_t(t_n)-u^n_t(t)\|_{H^{-1}}\le  C(R)|t^n-t|+CR^{\frac{m+1}{2}}|t^n-t|^{\frac12},
    \end{equation}
    which together with \eqref{u_cont} gives the desired continuity estimate.

    Next, we will prove that $\xi_{u^n}(t)\wto \xi_{u}(t)$ in $\H^{-1}$ for any $t\in [0,T]$ where $T>t^{*}$. According to Proposition \ref{partial_cont}, we have $u^n\to u$ in $C([0,T];L^2(\Omega))$, therefore it remains to verify the weak convergence of $u^n_t(t)$ in $H^{-1}$.

    According to equation \eqref{barw},
    \begin{equation}\label{unt}
        u_t^n(t)=w_{tt}^n(t)=\Delta w^n(t)-\int_0^t f(u^n(\tau))\td \tau-\Sigma(u^n(t))+\Sigma(u^n_0)+u_1^n+t\phi.
    \end{equation}
    Let us deal with \eqref{unt} term by term. Fisrtly, thanks to $u^n\to u$ in $C([0,T];L^2(\Omega))$ and $\{u^n\}$ be bounded in $L^{\infty}(0,T;H^1_0(\Omega))$, by a density agrument, we can obtain that
    \begin{equation}\label{dwn}
        \Delta w^n(t)  \wto \Delta w(t) \bin  H^{-1}.
    \end{equation}Since $z^n\to z$ in $\H$, by Sobolev embedding inequality, we get convergence for terms involving initial data
    \begin{equation}\label{inti}
        \Sigma(u^n_0) \to \Sigma(u_0) \bin    H^{-1} \and   u_1^n \to u_1 \bin L^2(\Omega).
    \end{equation}
    To deal with convergence of $\int_0^t f(u^n(\tau))\td \tau$ and $\Sigma(u^n(t))$, combining the $L^2$ convergence
    $$u^n\to u \bin C([0,T];L^2(\Omega))$$
    with the $H^1_0(\Om)$ uniform boundness $\sup_{n,t}\|u^n(t)\|_{H^1_0(\Om)}<+\infty$, we can deduce that $u^n\to u$ in $C([0,T];L^{\frac{11}{2}}(\Omega))$. Then for any $t<T$, by a dense agrument we can prove that
    \begin{equation}\label{f_sigma}
        \int_0^t f(u^n(\tau))\td \tau  \wto \int_0^t f(u(\tau))\td \tau \text{~~and~~} \Sigma(u^n(t))  \wto \Sigma(u(t))  \text{~~in~~}  H^{-1}.
    \end{equation}
    Combining with all the convergences \eqref{dwn}-\eqref{f_sigma}, we get
    \begin{equation}
        u^n_t(t)  \wto u_t(t) \bin  H^{-1} \text{~~for any ~~}t<T.
    \end{equation}
    Therefore,
    $$\xi_{u^n}(t^{*})\wto \xi_{u}(t^{*}) \bin H^{-1}.$$

    In conclusion, we have the weak convergence
    \begin{equation}
        \xi_{u^n}(t_n)=(\xi_{u^n}(t_n)-\xi_{u^n}(t^{*}))+\xi_{u^n}(t^{*})\wto \xi_{u}(t^{*}) \bin H^{-1}.
    \end{equation}
\end{proof}
\section{Global attractor}
In this section, we will prove the existence of global attractor for problem \eqref{eq1}. For the convenience of the reader, we now recall the definition of the global attractor, see \cite{babin92, temam97} for more details.
\begin{Def}
    Let $\{S(t)\}_{t\geq 0}$ be a semigroup on a metric space $(X,d)$.
    A subset $\mathscr{A}$ of $X$ is called a global attractor for the
    semigroup, if $\mathscr{A}$ is compact and enjoys the following properties:\\
    (1) $\mathscr{A}$ is invariant , i.e. $S(t)\mathscr{A}=\mathscr{A}$,~$\forall t\geq 0$;\\
    (2) $\mathscr{A}$ attracts all bounded sets of $X$, that is, for any bounded
    subset $B$ of $X$,
    $$ dist(S(t)B,\mathscr{A})\rightarrow 0,~~~as ~~t\rightarrow 0 ,$$
    where $dist(B,A)$ is the semidistance of two sets B and
    A:$$dist(B,A)=\sup \limits_{x\in B}\inf\limits_{y\in A}d(x,y).$$
\end{Def}
We will use the following existence criterion of global attractor for the norm-to-weak continuous semigroup.
\begin{thm}[\cite{zhong}, Theorem 4.2, Remark 4.4]
    Let $X$ be a Banach space and $\{S(t)\}_{t\ge 0}$ be a norm-to-weak continuous semigroup on $X$. Then $\{S(t)\}_{t\ge 0}$ has a global attractor in $X$, if and only if
    \begin{itemize}
        \item[(i)] $\{S(t)\}_{t\ge 0}$ has a bounded absorbing set $B_0$ in $X$;
        \item[(ii)] $\{S(t)\}_{t\ge 0}$ is asymptotically compact.
    \end{itemize}
\end{thm}

\subsection{Asymptotic compactness}



\begin{lemma}[Asymptotic compactness]\label{le4.6}
    Let the assumptions of Theorem \ref{weak} hold. Then the semigroup associated with \eqref{eq1} is asymptotically compact, that is for every sequence $\{z_n\}^{\infty}_{n=1}\subset \mathscr{B}_0$, and every sequence of times $t_n\to \infty$, there exists a subsequence $\{n_k\}$ such that
    $$\{S(t_{n_k})z_{n_k}\}^{\infty}_{=1}\rightarrow z \text{~~strongly in~~} \H.$$
\end{lemma}
\begin{proof}
    We will verify it by the energy method (\cite{ball04}, \cite{mrw98}). Denote
    \begin{equation}\label{eq4.9}
        (v_n(t),\p_t v_n(t))=
        \begin{cases}
            S(t+t_n)z_n, & t\ge -t_n \\
            0,           & t<-t_n.
        \end{cases}
    \end{equation}
    $\{(v_n(t),\p_t v_n(t))\}$ is bounded in $L^\infty(\R;\H)$, according to Banach-Alaoglu theorem, we may assume without loss of generality that there exists a sequence $\{y_k\}\in \H$ and $y\in L^\infty(\R;\H)$, such that
    \begin{equation}\label{eq4.10}
        \begin{split}
            (v_n,\p_t v_n)&\wto y\text{~weakly star in~}L^\infty(\R;\H),\\
            (v_n(-k),\p_t v_n(-k))&\wto y_k\text{~weakly in~}\H,~\forall k\in\mathbb{N}.
        \end{split}
    \end{equation}
    For fixed $k\in\mathbb{N}$, according to weakly continuous of $S(t)$, we have
    $$(v_n(t),\p_t v_n(t))=S(t+k)(v_n(-k),\p_t v_n(-k))\wto S(t+k)y_k \text{~~in~~} \H$$
    for any $t\ge -k$.

    The uniqueness of limit implies that $y(t)=S(t+k)y_k$, $t\ge -k$. Hence $y(t)$ is a weak solution of \eqref{eq1} with $y(-k)=y_k$ for any $k\in \mathbb{N}$. Since $k$ is arbitrary, therefore $(v(t),\p_t v(t))=y(t)$ is a complete solution.

    Rewriting equation  as follows
    \begin{equation}\label{eq4.11}
        \frac{\td }{\td t}E_\al(u(t))+ \al E_\al(u(t))+G_\al(u)+N_\al(u)+\Phi_\al(u)=0.
    \end{equation}
    where
    \begin{align*}
         & E_\al(u(t))=E(u(t))+\al\langle u(t),u_t(t)\rangle,                                                \\
         & G_\al(u)=\int_\Omega \left(\sigma(u)|u_t|^2-\frac{3\al}2 u_t^2+\alpha \sigma(u)u_t u\right)\td x, \\
         & N_\al(u)=\frac\al2\|\nabla u\|^2-\al^2\langle u,u_t\rangle,                                       \\
         & \Phi_\al(u)=\al\int_{\Omega} (f(u)u-F(u))\td x.
    \end{align*}
    Multiplying identity \eqref{eq4.11} for $v_n$ by $e^{\al t}$ and integrating over $[-t_n,0]$, we conclude that
    \begin{equation}\label{eq4.12}
        E_\al(v_n(0))+\int_{-t_n}^0e^{\al s}\left(G_\al(v_n(s))+N_\al(v_n(s))+\Phi_\al(v_n(s))\right)\td s=E_\al(z_n)e^{-\al t_n}.
    \end{equation}
    In order to pass the limit $n\to \infty$, we will deal with term by term.

    Firstly, the Aubin-Lions lemma implies that
    $$v_n\to v \text{~in~} C_{loc}(\R;L^5(\Omega)).$$
    Thus, $v_n\to v$ almost everywhere in $(0,\infty)\times \Omega$.   Since $F(v_n)$ is bounded from below and weakly lower semi-continuous,   $E_\alpha$ and $N_\al$ are weakly lower semi-continuous on $H^1_0(\Omega)$.
    From the weak convergence
    $$(v_n(0),\p_tv_n(0))=S(t_n)z_n\wto y(0)=(v(0),\p_tv(0))\text{~in~}\H,$$
    we obtain
    \begin{equation}\label{eq4.13}
        \varliminf_{n\to \infty}E_\al(v_n(0))\ge E_\al(v(0)),~~ \varliminf_{n\to \infty}N_\al(v_n(0))\ge N_\al(v(0)).
    \end{equation}

    For the term involving $\Phi_\al(v_n)$ in \eqref{eq4.12},  notice that
    due to \eqref{F3}, $f(s)s-F(s)+K s^2$ is bounded from below. Thus Fatou's Lemma implies that
    \begin{equation*}
        \begin{aligned}
            \varliminf_{n\to\infty}\int_{-t_n}^0e^{\al s} \left(\Phi_\al(v_n)+\int_\Omega K v_n^2\td x\right) \td s \ge  \int_{-\infty}^0e^{\al s} \left(\Phi_\al(v)+\int_\Omega K v^2\td x\right) \td s.
        \end{aligned}
    \end{equation*}
    Meanwhile, $\|v_n\|^2 \to \|v\|^2$ in $C_{loc}(\R)$ yields
    \begin{equation*}
        \lim_{n\to\infty}\int_{-t_n}^0e^{\al s}\int_\Omega v_n^2\td x \td s= \int_{-\infty}^0e^{\al s}\int_\Omega v^2\td x \td s.
    \end{equation*}
    Therefore,
    \begin{equation}\label{new:20}
        \begin{aligned}
            \varliminf_{n\to\infty}\int_{-t_n}^0e^{\al s} \Phi_\al(v_n) \td s\ge \int_{-\infty}^0e^{\al s} \Phi_\al(v)\td s.
        \end{aligned}
    \end{equation}

    Next, to estimate the term involving $G_\al$ in \eqref{eq4.12}, recall that
    \begin{equation} \label{newdef:G}
        G_\al(v_n)=\int_\Omega \left(\sigma(v_n)|\partial_t v_n|^2-\frac{3\al}2 |\partial_t v_n|^2\right)\td x+\alpha \int_\Omega\sigma(v_n)\partial_t v_n v_n \td x.
    \end{equation}
    For the first term in \eqref{newdef:G}, since $\sigma(u)\ge 0$, we have
    \begin{equation}\label{eq4.16}
        \begin{aligned}
              & \varliminf_{n\to\infty}\int_{-t_n}^0e^{\al s} \int_\Omega  \sigma(v_n)|\partial_t v_n|^2 \td x\td s-\int_{-t_n}^0e^{\al s} \int_\Omega  \sigma(v)|\partial_t v|^2 \td x\td s \\
            = & \varliminf_{n\to\infty}\int_{-t_n}^0e^{\al s} \int_\Omega  \sigma(v_n)|\partial_t v_n|^2 \td x\td s+\int_{-t_n}^0e^{\al s} \int_\Omega  \sigma(v)|\partial_t v|^2 \td x\td s \\
              & -2\lim_{n\to\infty}\int_{-t_n}^0e^{\al s} \int_\Omega  \sqrt{\sigma(v_n)}\partial_t v_n\sqrt{\sigma(v)}\partial_t v \td x\td s                                               \\
            = & \varliminf_{n\to\infty}\int_{-t_n}^0e^{\al s}\int_\Omega  |\sqrt{\sigma(v_n)}\partial_t v_n-\sqrt{\sigma(v)}\partial_t v|^2 \td x\td s\ge 0.
        \end{aligned}
    \end{equation}
    For the second term in \eqref{newdef:G}, since
    $$ |\partial_t v_n|^2\le \frac1{\sigma_0} \sigma(v_n)|\partial_t v_n|^2,$$
    which implies that
    \begin{equation}
        \begin{aligned}
            \int_{-t_n}^0e^{\al s}\int_\Omega |\partial_t v_n|^2\td x\td s \le \frac1{\sigma_0}\int_{-t_n}^0e^{\al s}\int_\Omega  \sigma(v_n)|\partial_t v_n|^2\td x \td s.
        \end{aligned}
    \end{equation}
    For the term involving $\sigma(v_n)\partial_t v_n v_n$ in \eqref{newdef:G}, using condition \eqref{S2}
    \begin{equation}\label{eq4.18}
        \begin{aligned}
             & \int_{-t_n}^0e^{\al s}\alpha\int_\Omega  \sigma(v_n)\left|\partial_t v_n v_n\right|\td x\td s                                                        \\
             & \le \int_{-t_n}^0e^{\al s}\al \left(\int_\Omega \sigma(v_n)|\partial_t v_n|^2\tdx\int_\Omega \sigma(v_n)|v_n|^2 \td x\right)^{\frac12}\td s          \\
             & \le \int_{-t_n}^0e^{\al s}\al \left(\int_\Omega \sigma(v_n)|\partial_t v_n|^2\tdx\int_\Omega C(1+|v_n|^6) \td x\right)^{\frac12}\td s                \\
             & \le \|\mathscr{B}_0\|_{\H} \int_{-t_n}^0e^{\al s}\al \left(\int_\Omega \sigma(v_n)|\partial_t v_n|^2\tdx\right)^{\frac12}\td s                       \\
             & \le \|\mathscr{B}_0\|_{\H}  \int_{-t_n}^0e^{\al s}\al\left( 4\delta^{-1}\int_\Omega \sigma(v_n)|\partial_t v_n|^2\tdx  +\delta\right)\td s           \\
             & = 4\al\delta^{-1} \|\mathscr{B}_0\|_{\H} \int_{-t_n}^0e^{\al s} \int_\Omega \sigma(v_n)|\partial_t v_n|^2 \td x\td s +\delta \|\mathscr{B}_0\|_{\H}.
        \end{aligned}
    \end{equation}
    Collecting \eqref{eq4.16} -- \eqref{eq4.18} and inserting into \eqref{newdef:G} we conclude that
    \begin{equation}
        \varliminf_{n\to\infty}\int_{-t_n}^0e^{\al s}G_\al(v_n)\td s \ge (1-\al a(\delta))\varliminf_{n\to\infty}\int_{-t_n}^0e^{\al s}\int_\Omega \sigma(v_n)|\partial_t v_n|^2 \td x\td s-b(\delta),
    \end{equation}
    where
    $$a(\delta)=\frac{3}{2\sigma_0}+4\delta^{-1} \|\mathscr{B}_0\|_{\H}$$
    and
    $$b(\delta)= \delta  \|\mathscr{B}_0\|_{\H}.$$
    Now, let $\al$ be small enough such that $1-\al a(\delta)>0$.  It then follows from \eqref{eq4.16}  that
    \begin{equation}\label{eq4.19}
        \begin{aligned}
            \varliminf_{n\to\infty}\int_{-t_n}^0e^{\al s}G_\al(v_n)\td s & \ge (1-\al a(\delta))\varliminf_{n\to\infty}\int_{-t_n}^0e^{\al s}\int_\Omega \sigma(v_n)|\partial_t v_n|^2 \td x\td s-b(\delta) \\
                                                                         & \ge (1-\al a(\delta))\int_{-\infty}^0e^{\al s}\int_\Omega \sigma(v)|\partial_t v|^2 \td x\td s-b(\delta).
        \end{aligned}
    \end{equation}

    Repeating the derivation of inequality \eqref{eq4.18} with $v_n$ replaced by $v$, we obtain:
    \begin{equation}\label{eq4.20}
        \begin{aligned}
           &\int_{-\infty}^0e^{\al s}\left(\int_\Omega  \sigma(v)|\partial_t v|^2\td x-G_\al(v)\right) \td s\\  
           \ge& -\al\int_{-\infty}^0e^{\al s}\int_\Omega \sigma(v)\partial_t v v\td x\td s                           \\
            \ge& -\al a(\delta) \int_{-\infty}^0e^{\al s}\int_\Omega  \sigma(v)|\partial_t v|^2 \td x\td s-b(\delta).
        \end{aligned}
    \end{equation}
    Substituting \eqref{eq4.20} into \eqref{eq4.19}, we get:
    \begin{equation}\label{eq4.22}
        \begin{aligned}
        &\varliminf_{n\to\infty}\int_{-t_n}^0e^{\al s}G_\al(v_n)\td s \\
        \ge& \int_{-\infty}^0e^{\al s} G_\al(v) \td s -2\al a(\delta)\int_{-\infty}^0e^{\al s}\int_\Omega  \sigma(v)|\partial_t v|^2 \td x\td s-2b(\delta).
        \end{aligned}
    \end{equation}
    Combining with inequalities \eqref{eq4.13}, \eqref{new:20} and \eqref{eq4.22} gives
    \begin{equation}\label{eq4.23}
        \begin{aligned}
                & \varliminf_{n\to \infty}  \int_{-t_n}^0e^{\al s}\left(G_\al(v_n(s))\td s+N_\al(v_n(s)) \td s + \Phi_\al(v_n(s))\right)\td s                                                 \\
            \ge & \int_{-\infty}^0e^{\al s}\left(G_\al(v)+\Phi_\al(v)+N_\al(v)\right)\td s\\
            &-2\al a(\delta)\int_{-\infty}^0e^{\al s}\int_\Omega \sigma(v)|\partial_t v|^2\td x\td s-2b(\delta).
        \end{aligned}
    \end{equation}
    On the other hand, since $v(t)$ is also a weak solution of equation \eqref{eq1} on $\R$, repeating the derivation of \eqref{eq4.12}, we get
    \begin{equation}
        E_\al(v(0))+\int_{-\infty}^0e^{\al s}\left(G_\al(v)+\Phi_\al(v)+N_\al(v)\right)\td s=0.
    \end{equation}
    Thus passing the lower limit as $n\to \infty$ in the equality \eqref{eq4.12}  yields
    \begin{equation*}
        \begin{aligned}
            0  = & \varliminf_{n\to \infty} \left(E_\al(v_n(0))+\int_{-t_n}^0e^{\al s}\left(G_\al(v_n(s)) + N_\al(v_n(s))+\Phi_\al(v_n(s))\right)\td s\right)             \\
            \ge  & \varliminf_{n\to \infty}  E_\al(v_n(0))+\int_{-\infty}^0e^{\al s}\left(G_\al(v)+\Phi_\al(v)+N_\al(v)\right)\td s                                       \\
                 & -2\al a(\delta)\int_{-\infty}^0e^{\al s}\int_\Omega \sigma(v)|\partial_t v|^2\td x\td s-2b(\delta)                                                     \\
            =    & \varliminf_{n\to \infty}  E_\al(v_n(0))-E_\al(v(0))-2\al a(\delta)\int_{-\infty}^0e^{\al s}\int_\Omega \sigma(v)|\partial_t v|^2\td x\td s-2b(\delta).
        \end{aligned}
    \end{equation*}
    In conclusion, the following inequality holds:
    \begin{equation*}
        E_\al(v(0))\le \varliminf_{n\to \infty}E_\al(v_n(0)) \le E_\al(v(0))+2\al a(\delta)\int_{-\infty}^0e^{\al s}\int_\Omega \sigma(v)|\partial_t v|^2\td x\td s+2b(\delta).
    \end{equation*}
    Note that $\int_{-\infty}^0 \int_\Omega \sigma(v)|\partial_t v|^2\td x\td s$ is finite, so we can take the limit $\alpha \to 0$ to derive
    \begin{equation} \label{new:23}
        E(v(0))\le \varliminf_{n\to \infty}E(v_n(0)) \le E(v(0))+2b(\delta).
    \end{equation}
    Since \eqref{new:23} holds for any $\delta>0$, passing the limit $\delta\to 0$, we obtain the convergence of energy
    \begin{equation*}
        \varliminf_{n\to \infty}E(v_n(0))=E(v(0)),
    \end{equation*}
    which implies that $$\varliminf_{n\to \infty}\|S(t_n)z_n\|_{\H}=\|y(0)\|_{\H}.$$ Recall that weak convergence $S(t_n)z_n\wto y(0)\text{~in~}\H$, therefore there exists a strong convergent subsequence $S(t_n)z_n\to y(0)$ strongly in $\H$. Thus the asymptotic compactness of the semigroup $S(t)$ is verified. The proof is complete.
\end{proof}

\subsection{Exsitence of the global attractor}
Now we are ready to state our main result in this section.
\begin{thm}\label{th4.4}
    Let the assumptions of Theorem \ref{weak} hold. Then, the solution semigroup $S(t)$ generated by the weak solution of problem \eqref{eq1} possesses a global attractor $\mr{A}$ in the phase space $\H$.
\end{thm}
\begin{proof}
    From Lemma 4.3 and Lemma 4.5, it is clear that the conclusion is valid.
\end{proof}

\section{Regularity of the global attractor}
In this section, we will establish the regularity of the global attractor $\mr{A}$ for the semigroup $\{S(t)\}_{t\ge 0}$ generated by the weak solution of problem \eqref{eq1}. To this end, there is an additional condition of the damping coefficient $\sigma=\sigma_l$, namely for some large enough constant $l>0$ , it holds that
\begin{equation}\label{S3}
    \sigma_l(s)=\gamma>0,\forall s\in [-l,l]\tag{S3}
\end{equation}

Here, we state the main result as follow.
\begin{thm}[Regularity of the global attractor]\label{regularity}
    Assume that $\phi\in L^2(\Omega)$ and conditions \eqref{S1}-\eqref{F2} and \eqref{S3} hold with damping parameter $l$ large enough. Then the global attractor $\mr{A}$ of problem \eqref{eq1} is bounded in $\V$.
\end{thm}

The regularity of the global attractor will be demonstrated in two steps. First, a uniform estimate of the strong solution is established in Lemma \ref{unif_strogn_est} below. Next, find a regular point in $\V$ on any bounded complete trajectory, which are shown in Lemma \ref{lem5.3} and Lemma \ref{lem5.4}.

{\bf Step 1: Uniform energy estimate of the strong solution.}

To prove the regularity of A, we start with the following lemma:

\begin{lemma}\label{unif_strogn_est}
    Assume that $\phi\in L^2(\Omega)$ and conditions \eqref{S1}-\eqref{F2} and \eqref{S3} hold. Let $B$ be a bounded subset of $\V$. Then there exists $l(B) > 0$ such that for every
    $l\ge l_B$ and $(u_0, u_1)\in B$ the problem \eqref{eq1} has a unique strong solution $u\in L^\infty(0,\infty;H^2(\Omega))\cap W^{1,\infty}(0,\infty;H^1_0(\Omega))$, which is also the solution of
    \begin{equation}\label{eq_linear}
        \begin{cases}
            v_{tt}+\gamma v_t-\Delta v+ f(v) = \phi, in~ (0,\infty)\times\Omega, \\
            v = 0, on~ (0,\infty)\times\partial \Omega,                          \\
            v(0,\cdot) = u_0, v_t(0,\cdot) = u_1, in~ \Omega,
        \end{cases}
    \end{equation}
    and satisfies the following inequality
    \begin{equation}\label{}
        \|(u(t),u_t(t))\|_{\V}\le r, ~~\forall t\ge 0,
    \end{equation}
    where the positive constant $r$ depends on $B$ and $\phi$, but is independent of $(u_0, u_1)$.
\end{lemma}
\begin{proof}
    The proof follows the strategy in \cite{khan10-3d}. According to \cite{yan}, problem \eqref{eq_linear} has a unique strong solution and there exists a constant $r_1$, which depends on $B$, such that
    $$\|(v(t),v_t(t))\|_{\V}\le r_1.$$
    Then by Sobolev embedding $H^2(\Omega)\subset C(\bar{\Omega})$,
    \begin{equation}\label{L_infty}
        \|v(t)\|_{C(\bar{\Omega})}\le r_2.
    \end{equation}
    Now let $l(B)=r_2+1$, consider problem \eqref{eq1} for $l\ge l_B$. To derive the bound for the strong solution of equation \eqref{eq1}, multiply equation \eqref{eq1} by $-\Delta u_t$. However, since $-\Delta u_t$ belongs to $H^{-1}$, it cannot directly serve as a test function. Therefore, a formal derivation of the a priori estimate is presented next, which is justified using the Galerkin approximation method described in Theorem \ref{weak}.

    Multiplying equation \eqref{eq1} by $-\Delta u_t$ and integrating on $\Omega$, we have
    \begin{equation}\label{Dut}
        \frac{\td}{\td t}\left(\frac12\|\nabla u_t\|^2+\frac12\|\Delta u\|^2\right)-\int_\Omega \sigma_l(u)u_t\Delta u_t\td x-\int_\Omega f(u)\Delta u_t\td x=0.
    \end{equation}
    For the term involves damping,
    \begin{equation}\label{sigma_ut_Dut}
        \begin{aligned}
            -\int_\Omega \sigma_l(u)u_t\Delta u_t\td x & =\int_\Omega \sigma_l(u)|\nabla u_t|^2\td x-\int_\Omega \sigma_l^\prime(u)u_t\nabla u\cdot\nabla u_t\td x \\
                                                       & \ge -\|\sigma_l^\prime(u)\|_{L^6(\Omega)}\|u_t\|_{L^6(\Omega)}\|\nabla u\|_{L^6(\Omega)}\|\nabla u_t\|    \\
                                                       & \ge -C(1+\|\Delta u\|^{m-1})\|\nabla u_t\|^2\|\Delta u\|                                                  \\
                                                       & \ge -C((\|\nabla u_t\|^2+\|\Delta u\|^2)^3+1).
        \end{aligned}
    \end{equation}
    And for the last term in \eqref{Dut}, we have
    \begin{equation}\label{f_Dut}
        \begin{aligned}
            -\int_\Omega f(u)\Delta u_t\td x & =\int_\Omega f^\prime(u)\nabla u\cdot \nabla u_t\td x\ge -\|f^\prime(u)\|_{L^{3}(\Omega)}\|\nabla u\|_{L^6(\Omega)}\|\nabla u_t\| \\
                                             & \ge -C(1+\|\Delta u\|^{p-1})\|\Delta u\|\|\nabla u_t\|                                                                            \\
                                             & \ge -C((\|\nabla u_t\|^2+\|\Delta u\|^2)^3+1).
        \end{aligned}
    \end{equation}
    Substituting \eqref{sigma_ut_Dut} and \eqref{f_Dut} into \eqref{Dut}, we have
    \begin{equation}
        \frac{\td}{\td t}\left( \|\nabla u_t\|^2+ \|\Delta u\|^2\right)\le C((\|\nabla u_t\|^2+\|\Delta u\|^2)^3+1).
    \end{equation}

    We see that problem \eqref{eq1} with initial data $(u_0, u_1)$, has a unique local strong solution $u\in L^\infty(0,T_B;H^2(\Omega))\cap W^{1,\infty}(0,T_B;H^1_0(\Omega))$ for some $T_B=T(\|B\|_{\V})> 0$.
    Then by the embedding theorems (see \cite[Theorem 3.1, p. 19 and Lemma 8.1, p. 275]{lions72}), we obtain that $(u, u_t ) \in C_s(0, T_B; H^1_0(\Omega))$ and $u \in C([0, T_B]; C(\bar{\Omega}))$.

    Set
    $$T_{sup} = \sup \{T  :  u \in L^{\infty}(0, T; H^2(\Omega)) \cap W^{1,\infty} (0, T; H^1_0(\Omega))\}.$$
    Let us show that $T_{sup} =\infty$. If not, then
    $$T_{sup}<\infty \and \limsup_{t\to T_{sup}}\|(u(t), u_t (t))\|_{\V}=\infty.$$
    Assume that $\limsup_{t\to T_{sup}}\|u(t)\|_{C(\bar{\Omega})}>r_2$. Then, since by \eqref{L_infty}, $\|u_0\|_{C(\bar{\Omega})} \le r_2$, there exists $\varepsilon \in (0, 1)$ and $T_0\in (0, T_{sup})$, such that  $\|u(T_0)\|_{C(\bar{\Omega})}= r_2 + \varepsilon$ and $\|u(t)\|_{C(\bar{\Omega})}\le r_2 + \varepsilon$, $t\in [0, T_0]$. So by \eqref{S3}, for $l\ge l(B)$ the function $u(t, x)$ satisfies the linear damping problem \eqref{eq_linear} on $[0, T_0]\times \Omega$ and by uniqueness we have $u(t) = v(t)$ on $[0, T_0]$. Then by \eqref{L_infty}, we have $\|u\|_{C([0,T_0];C(\bar{\Omega}))} \le r_2$, which contradicts to $\|u(T_0)\|_{C(\bar{\Omega})}= r_2 + \varepsilon$. Hence $T_{sup} =\infty$.
\end{proof}


{\bf Step 2: Regularity estimate in $\V$ of the global attractor.}

Here we will follow the cutting-off method of Khanmamedov\cite{khanmamedov10}, which has been refined in \cite{liu23} to establish regularity estimate in $\V$ of global attractor for wave equation with nonlinear interior damping term $g(u_t)$.

Since $\mathscr{A} = S(t)\mathscr{A}$ for all $t > 0$, for any point in the attractor $(u_0, u_1)\in \mathscr{A}$, there is a trajectory $z(t)$ passing trough this point and such that $z(t)=(u(t), u_t(t))\in \mathscr{A}$ for all $t\in R$. Since $\mathscr{A}$ is compact, there exist a decreasing negative sequence $\{t_{i}\}$, $t_{i}\to -\infty$ and a point $z_0=(\bar{u}_0, \bar{u}_1)\in \mathscr{A}$, s.t. $z(t_{i})\to z_0$ in $\H$ as $i\to \infty$. Then $z(t+t_{i})=S(t)z(t_{i})\to S(t)z_0$ for any $t\ge 0$.

On the other hand, the dynamical system $(S(t), \H)$ is a gradient system with a strict Lyapunov functional
$$\mathcal{L}(z(t))=\frac12\|u_t\|^2+\frac12\|\nabla u\|^2+\int_\Omega F(u)\td x-\int_\Omega \phi u\td x.$$
It concludes that
$$\mathcal{L}(z_0)=\lim_{i\to\infty}\mathcal{L}(z(t_{i}))=\lim_{i\to\infty}\mathcal{L}(z(t+t_{i}))=\mathcal{L}(S(t)z_0)~~ \text{for any  } t\ge 0,$$
and therefore $z_0$ is an equilibrium point of semigroup $S(t)$, i.e., $\bar{u}_0$ satisfies the elliptic equation $-\Delta \bar{u}_0+f(\bar{u}_0)=\phi$ and $\bar{u}_1=0$. It is easy to see that the set of equilibrium points of semigroup $S(t)$ is bounded in $\V$.

Let $g(s)=f(s)+Ks$, then $g^{\prime}(s)\ge 0$ and $g(s)$ is nondecreasing. Define finite cutting-off $g_n(s)$ as
\begin{equation}\label{gn}
    g_n(s)=\begin{cases}
        g(n),~~ s\geq n, \\
        g(s),~~ |s|<n,   \\
        g(-n),~~ s\le -n,
    \end{cases}
    n\in\mathbb{N}.
\end{equation}
Define the primitive function $G(s)=\int_0^s g(\tau)\td \tau$ and $G_n(s)=\int_0^s g_n(\tau)\td \tau$. Then it is easy to verify that the following inequalities hold:
$$g^{\prime}(s)-g^{\prime}_n(s)\ge 0,\quad |g(s)-g_n(s)|\le C|s|^5 \and 0\le G(s)-G_n(s)\le C|s|^6.$$

Next, we introduce the decomposition $u(t)=w_n(t)+v_n(t)+\bar{u}_0$, where $w_n(t)$ satisfies following weak damped wave equation with zero initial data
\begin{equation}\label{wn}
    \begin{cases}
        \partial_t^2 w_n+\gamma\partial_t w_n-\Delta w_n+g_n(u)-f(\bar{u}_0)-Ku=0,~~\text{in }(t_{i},\infty)\times\Omega, \\
        w_n=0,~~\text{on }(t_{i},\infty)\times \partial\Omega,                                                            \\
        w_n(t_{i},x)=0,~~ \partial_{t} w_n(t_{i},x)=0,~~\text{in }\Omega,
    \end{cases}
\end{equation}
and remainder $v_n(t)$ satisfies
\begin{equation}\label{vn}
    \begin{cases}
        \partial_t^2 v_n+\sigma_l(u)\partial_t u-\gamma\partial_t w_n-\Delta v_n+g(u)-g_n(u)=0,~~\text{in }(t_{i},\infty)\times\Omega, \\
        v_n=0,~~\text{on }(t_{i},\infty)\times\partial\Omega,                                                                          \\
        v_n(t_{i},x)=u(t_{i},x)-\bar{u}_0(x),~~ \partial_{t} v_n(t_{i},x)=u_t(t_{i},x),~~\text{in }\Omega.
    \end{cases}
\end{equation}

This decomposition and the corresponding boundary conditions facilitate the derivation of regularity estimates for the components $w_n$ and $v_n$, thereby establishing the desired regularity estimate in $\V$ for the global attractor in the sequel.

\begin{lemma}\label{lem5.3}
    Under the assumptions of Theorem \ref{regularity}, we have $(w_{n}(t), \partial_{t} w_n(t)) \in \V$ and for any $n\in \mathbb{N}$ there exists
    $T_n < 0$ such that
    \begin{equation}\label{wH2}
        \|\nabla \partial_{t} w_n(t)\|+\|\Delta w_{n}(t)\| \le  r_{0}n^{2}\|\nabla v_n(t)\|+r_0,~~\forall t_{i}\le  t\le  T_{n},
    \end{equation}
    where the positive constant $r_0$ is independent of $i$, $n$ and $(u_0, u_1)$.
\end{lemma}
\begin{proof}
    {\it(i) The $H^1_0(\Omega)\times H^1_0(\Omega)$ estimate.}

    Differentiating \eqref{wn} with respect to $t$, multiplying both sides of the derived equation by $\partial_t^2 w_n$ and integrating over $(t_{i}, T)\times \Omega$ implies
    \begin{equation}
        \begin{aligned}
                & \|\partial_t^2 w_n(t)\|^2+\|\nabla \partial_t w_{n}(t)\|^2                                                                                    \\
            \le & C\left(\|\partial_t^2 w_n(t_{i})\|^2+\|\nabla \partial_t w_{n}(t_{i})\|^2+\int_{t_{i}}^{t}\|(g_n^{\prime}(u(t))-K)u_t(t)\|^{2}\td t+1\right).
        \end{aligned}
    \end{equation}
    According to the definition of $g_n$ and growth condition \eqref{F1}, it follows that
    \begin{equation}\label{fnp}
        |g_n^{\prime}(u(t))|\le Cn^4,
    \end{equation}
    which yields
    \begin{equation}\label{5.19}
        \begin{aligned}
             & \|\partial_t^2 w_n(t)\|^2+\|\nabla \partial_t w_{n}(t)\|^2\le C\left(\|\partial_t^2 w_n(t_{i})\|^2+(n^{8}+K)\int_{t_{i}}^{t}\|u_{t}(t)\|^{2}\td t+1\right)
        \end{aligned}
    \end{equation}
    for any $t\ge t_{i}$.

    From the equation \eqref{wn}, since
    $$
        \begin{aligned}
            \partial_t^2 w_n(t_{i}) & =f(\bar{u}_0)-g_n(u(t_{i}))+Ku(t_{i})                                            \\
                                    & =g(\bar{u}_0)-g_n(u(t_{i}))+K(u(t_{i})-\bar{u}_0)                                \\
                                    & =g(\bar{u}_0)-g_n(\bar{u}_0)+g_n(\bar{u}_0)-g_n(u(t_{i}))+K(u(t_{i})-\bar{u}_0), \\
        \end{aligned}$$
    then
    \begin{equation}
        \begin{aligned}
            \|\partial_t^2 w_n(t_{i})\| & \le \|g(\bar{u}_0)-g_n(\bar{u}_0)\|+\|g_n(\bar{u}_0)-g_n(u(t_{i}))\|+K\|u(t_{i})-\bar{u}_0\| \\
                                        & \le Cn^2\|\bar{u}_0-u(t_{i})\|_6+C.
        \end{aligned}
    \end{equation}
    Since $\|\bar{u}_0-u(t_{i})\|_6\to 0$ as $i\to \infty$, then for any $n\in \mathbb{N}$ there exists $M_n\in \mathbb{N}$, such that $n^2\|\bar{u}_0-u(t_{i})\|_6\le 1$ for all $i>M_n$, therefore
    \begin{equation}\label{5.22}
        \|\partial_t^2 w_n(t_{i})\|\le C_1, ~~ \forall i>M_n,
    \end{equation}
    where the positive constant $C_1$ are independent of $i$, $n$ and $(u_0, u_1)$.

    Since $u(t)$ is a full time bounded solution, then by energy estimate \eqref{weak_energy_est}, we have the dissipativity relation
    $$\int_{-\infty}^{+\infty}\int_{\Omega}\sigma_l(u(t))|u_{t}|^2\td x\td t\le Q(\|\mathscr{A}\|_{\H}).$$
    Then by \eqref{S1}
    \begin{equation}\label{5.16}
        \int^{+\infty}_{-\infty}\|u_{t}(t)\|^{2}\td t\le \frac{1}{\sigma_0}\int_{-\infty}^{+\infty}\int_{\Omega}\sigma_l(u)|u_{t}|^2\td x\td t\le Q(\|\mathscr{A}\|_{\H}).
    \end{equation}
    Consequently, for any $n \in \mathbb{N}$ there exists $T_n<t_{M_n}<0$ such that
    \begin{equation}\label{5.24}
        n^{8}\int_{-\infty}^{T_n}\|u_{t}(t)\|^{2}\td t\le 1.
    \end{equation}
    Substituting \eqref{5.22} and \eqref{5.24} into \eqref{5.19}, one obatins
    \begin{equation}\label{}
        \|\partial_t^2 w_n(t)\|+\|\nabla \partial_t w_{n}(t)\| \le  R_{1},~~\text{for any  } t_{i}\le  t\le  T_n,
    \end{equation}
    where the positive constant $R_1$ are independent of $i$, $n$ and $(u_0, u_1)$.

        {\it(ii) $H^1_0(\Omega)$ estimate of $w_{n}(t)$.}

    Multiplying both sides of \eqref{wn} by $\partial_{t} w_n+\al w_n$, where $\al \in (0, 1)$, and integrating over $\Omega$, it follows that
    \begin{equation}\label{Delta_w}
        \begin{aligned}
             & \frac{\td}{\td t}\Phi_{n}(t) +\al\Phi_{n}(t)+(\gamma-\frac{3\al}{2})\|\partial_{t} w_n\|^2 +\frac\al2\|\nabla w_{n}\|^2                               \\
             & +(\al \gamma-\al^2)\int_\Omega \partial_{t} w_{n} w_{n}\td x-\int_\Omega g_n^\prime(u)u_t w_{n}\td x-\int_\Omega Ku(\partial_{t} w_n+\al w_n)\td x=0,
        \end{aligned}
    \end{equation}
    in which
    \begin{equation*}
        \Phi_{n}(t)={\frac{1}{2}}\| \partial_t w_{n}\|^{2}+{\frac{1}{2}}\|\nabla w_{n}\|^{2}+\al\int_\Omega\partial_{t}w_{n} w_{n}\td x+\int_\Omega g_n(u) w_{n}\td x-\int_\Omega f(\bar{u}_0) w_{n}\td x,
    \end{equation*}
    Since $u\in \mr{A}$ is bounded in $H^1_0(\Omega)$, one has
    \begin{equation}\label{5.34}
        \Phi_{n}(t)\ge \frac14\left(\|\partial_t w_{n}(t)\|^{2}+\|\nabla w_{n}(t)\|^{2}\right)-C_3,~~ t_{i}\le  t\le  T_{n}.
    \end{equation}
    Here, we just need the lower bound estimate of $\Phi_n(t)$ since $\Phi_n(t_{i})=0$ .

    Regarding the first term on the last line of \eqref{Delta_w}, the Cauchy inequality yields
    \begin{equation}\label{5.21}
        (\al \gamma-\al^2)\int_\Omega |w_{n} \partial_{t} w_n|\td x \le C\al(\gamma-\al)^2\int_\Omega |\partial_{t} w_n|^2\td x+\frac\al6\|\nabla w_{n}\|^2 .
    \end{equation}
    Applying H\"older's inequality to the second and fourth terms on the last line of \eqref{Delta_w} gives
    \begin{equation}\label{5.38}
        \left|\int_\Omega g_n(u)u_t w_{n}\td x\right| \le \|g_n(u)\|_{L^3(\Omega)}\|u_t\| \|w_{n}\|_{L^6(\Omega)}\le Cn^8\|u_t\|^2+\frac{\al}6\|\nabla w_{n}\|^2,
    \end{equation}
    and
    \begin{equation}
        \begin{aligned}
            \left|\int_\Omega Ku(\partial_{t} w_n+\al w_n)\td x\right| & \le CK\|u\|^2+\al\|\partial_{t} w_n\|^2+\frac{\al}6\|\nabla w_n\|^2 \\
                                                                       & \le  C+\al\|\partial_{t} w_n\|^2+\frac{\al}6\|\nabla w_n\|^2.
        \end{aligned}
    \end{equation}
    Choosing $\al$ be small enough in \eqref{5.38}, and combining with \eqref{Delta_w} and \eqref{5.21} results in
    \begin{equation}\label{}
        \begin{aligned}
            \frac{\td}{\td t} \Phi_{n}(t)+ \al \Phi_{n}(t)+\al \|\partial_{t} w_n\|^2 \le Cn^8\|u_t\|^2+C.
        \end{aligned}
    \end{equation}
    Taking into account \eqref{5.24}, it can be drived from Gronwall's Lemma that
    \begin{equation}\label{5.40}
        \Phi_{n}(t)+\int_{t_i}^{T_n}\|\partial_{t} w_n\|^2\td t\le C_4,~~ t_{i}\le  t \le  T_{n}.
    \end{equation}
    which together with \eqref{wn} yields
    \begin{equation}\label{5.26}
        \|\partial_t^2 w_n(t)\|+\|\nabla \partial_{t} w_n(t)\|+\|\nabla w_{n}(t)\| \le  R_2,~~\text{for any  } t_{i}\le  t\le  T_n,
    \end{equation}
    where the positive constant $R_2$ are independent of $i$, $n$ and $(u_0, u_1)$.

        {\it(iii) $H^2(\Omega)$ estimate of $w_{n}(t)$.}

    Rewrite the equation \eqref{wn} as $-\Delta w_n+g_n(u)=f(\bar{u}_0)+Ku-\partial_t^2 w_n-\gamma\partial_t w_n$.
    Notice that, by \eqref{5.26}, the terms on the RHS of the elliptic equation and also $f_n(u)$ are in $L^{\infty}(t_{i}, T_n;L^2(\Omega))$. Thus, by squaring both sides of the above equation and integrating on $\Omega$, it follows that
    \begin{equation}\label{5.27}
        \|\Delta w_n\|^2+\|g_n(u)\|^2+2\int_\Omega g_n^\prime(u_n)\nabla u\nabla w_n\td x=\|f(\bar{u}_0)+Ku-\partial_t^2 w_n-\gamma\partial_t w_n\|^2.
    \end{equation}
    Since
    \begin{equation}\label{5.28}
        \begin{aligned}
            \int_\Omega g_n^\prime(u_n)\nabla u\nabla w_n\td x & =\int_\Omega g_n^\prime(u_n)|\nabla w_n|^2\td x+\int_\Omega g_n^\prime(u_n)\nabla v_n\nabla w_n\td x \\
                                                               & \ge -C\|g_n^\prime(u_n)\|^2_{L^3(\Omega)}\|\nabla v_n\|^2-\frac12\|\Delta w_n\|^2                    \\
                                                               & \ge -Cn^4\|\nabla v_n\|^2-\frac12\|\Delta w_n\|^2.
        \end{aligned}
    \end{equation}
    Combining with \eqref{5.26}-\eqref{5.28}, yields
    $$\|\Delta w_{n}(t)\|^2\le Cn^4\|\nabla v_n(t)\|^2+C, ~~\text{for any  } t_{i}\le  t\le  T_n$$
    which finishes the proof.
\end{proof}

\begin{lemma}\label{lem5.4}
    Under the assumptions of Theorem \ref{regularity}, there exists a $n_0 \in \mathbb{N}$ with the following uniform decay property: for any damping parameter $l\ge n_0$, there holds
    \begin{equation}\label{5.41}
        \lim_{i\to \infty}\sup_{t_{i}\le t\le T_{n_0}}\left(\|\partial_t v_{n_0}(t)\|+\|\nabla v_{n_0}(t)\|\right)=0.
    \end{equation}
\end{lemma}
\begin{proof}
    We initiate our analysis by multiplying both sides of equation \eqref{vn} by the test function $\partial_{t} v_n + \alpha v_n$ with $\alpha \in (0,1)$, followed by integration over $\Omega$. This yields the fundamental energy identity:
    \begin{equation}\label{5.30}
        \begin{aligned}
             & {\frac{\td}{\td t}}\left(\frac12\|\partial_{t} v_n\|^2+\frac12\|\nabla v_{n}\|^2+\al \langle \partial_{t} v_n,v_{n}\rangle\right)+\al \|\nabla v_{n}\|^{2}-\al\| \partial_{t} v_n\|^{2} \\
             & +\int_\Omega \left(\sigma_l(u)u_t-\gamma \partial_t w_n\right)(\partial_{t} v_n+\al v_{n})\td x =\int_\Omega \(g_n(u)-g(u)\)\(\partial_{t} v_n + \al v_n\)\td x.
        \end{aligned}
    \end{equation}

    For the right-hand side terms involving $g_n(u)-g(u)$, we employ Sobolev embeddings:
    \begin{equation}\label{eq5.31}
        \begin{aligned}
             & \int_\Omega \(g_n(u)-g(u)\)\partial_{t} v_n\td x                                                                  
              =\int_\Omega \(g_n(u)-g(u)\)(u_t-\partial_{t} w_n)\td x                                                           \\
              =& \frac{\td}{\td t}\int_\Omega \(G_n(u)-G(u)\)\td x -\int_\Omega \(g_n(u)-g(u)\) \partial_{t} w_n\td x            \\
              \le& \frac{\td}{\td t}\int_\Omega \(G_n(u)-G(u)\)+ C\|g_n(u)-g(u)\|_{L^{\frac65}(\Omega)} \|\partial_{t} w_n\|_{6} \\
              \le& \frac{\td}{\td t}\int_\Omega \(G_n(u)-G(u)\)+ C\|g_n(u)-g(u)\|_{L^{\frac65}(\Omega)}.
        \end{aligned}
    \end{equation}
    Additionally, we estimate
    \begin{equation}\label{eq5.32}
        \begin{aligned}
            \al\int_\Omega \(g_n(u)-g(u)\) v_n & \le\al\|g_n(u)-g(u)\|_{L^{\frac65}(\Omega)}\|v_n(t)\|_6                      \\
                                               & \le C\|g_n(u)-g(u)\|^2_{L^{\frac65}(\Omega)}+\frac{\al}4\|\nabla v_n(t)\|^2.
        \end{aligned}
    \end{equation}
    We decompose the fourth left-hand term using the structure of $\sigma_l(u)$,
    \begin{equation}\label{eq5.33}
        \begin{aligned}
                 & \int_\Omega \left(\sigma_l(u)u_t-\gamma \partial_t w_n\right)\partial_{t} v_n\td x                                                       \\
            =    & \int_\Omega \left(\sigma_l(u)\partial_{t} v_n+(\sigma_l(u)-\gamma) \partial_t w_n\right)\partial_{t} v_n\td x                            \\
            =    & \int_\Omega \sigma_l(u)|\partial_{t} v_n|^2\td x+\int_\Omega \(\sigma_l(u)-\gamma\)\partial_t w_n \partial_{t} v_n\td x                  \\
            \geq & \int_\Omega \sigma_l(u)|\partial_{t} v_n|^2\td x-\frac{\sigma_0}{6\gamma}\int_\Omega |\sigma_l(u(t))-\gamma||\partial_{t} v_n(t)|^2\td x \\
                 & -C\|\sigma_l(u(t))-\gamma\|_{L^{\frac32}(\Omega)} \|\partial_t w_n(t)\|^2_{L^6(\Omega)}                                                  \\
            \geq & \frac34\int_\Omega \sigma_l(u)|\partial_{t} v_n|^2\td x-C\|\sigma_l(u)-\gamma\|_{L^{\frac32}(\Omega)}.
        \end{aligned}
    \end{equation}
    Next, consider the contribution from $v_n$,
    \begin{equation}\label{eq5.34}
        \begin{aligned}
                & \int_\Omega \left(\sigma_l(u)u_t-\gamma \partial_t w_n\right)v_{n}\td x                                                   \\
            =   & \int_\Omega \left(\sigma_l(u)\partial_{t} v_n+(\sigma_l(u)-\gamma) \partial_t w_n\right)v_n\td x                          \\
            =   & \int_\Omega \sigma_l(u)\partial_{t} v_n v_n\td x+\int_\Omega \(\sigma_l(u)-\gamma\)\partial_{t} w_n v_n\td x              \\
            \ge & -C\int_\Omega \sigma_l(u(t))|\partial_{t} v_n(t)|^2\td x-C\int_\Omega |\sigma_l(u)-\gamma||v_n|^2\td x                    \\
                & -\|\sigma_l(u)-\gamma\|_{L^{\frac32}(\Omega)}\|\partial_{t} w_n\|_{L^6(\Omega)}-\frac{\lambda_1}{4}\|v_n\|^2              \\
            \ge & -C\int_\Omega \sigma_l(u)|\partial_{t} v_n|^2\td x-C\|\sigma_l(u)-\gamma\|_{L^{\frac32}(\Omega)}-\frac14\|\nabla v_n\|^2.
        \end{aligned}
    \end{equation}
    Combining \eqref{eq5.33} and \eqref{eq5.34} with parameter $\alpha$ be chosen small enough yields,
    \begin{equation}\label{eq5.36}
        \begin{aligned}
                & \int_\Omega \left(\sigma_l(u)u_t-\gamma \partial_t w_n\right)(\partial_{t} v_n+\al v_{n})\td x                          \\
            \ge & \frac34\int_\Omega \sigma_l(u)|\partial_{t} v_n|^2\td x-C\|\sigma_l(u)-\gamma\|_{L^{\frac32}(\Omega)}                   \\
                & -\al C\int_\Omega \sigma_l(u)|\partial_{t} v_n|^2\td x-\frac\al4\|\nabla v_n\|^2                                        \\
            \ge & \frac\al2\int_\Omega |\partial_{t} v_n|^2\td x-C\|\sigma_l(u)-\gamma\|_{L^{\frac32}(\Omega)}-\frac\al4\|\nabla v_n\|^2.
        \end{aligned}
    \end{equation}
    Substituting \eqref{eq5.31} \eqref{eq5.32} and \eqref{eq5.36} into \eqref{5.30}, we derive
    \begin{equation}\label{5.44}
        \begin{aligned}
             & {\frac{\td}{\td t}}\Psi_\al(t)+\al \Psi_\al(t)\le  N_n(t),
        \end{aligned}
    \end{equation}
    where the energy functional $\Psi_\al(t)$ and perturbation term $N_n(t)$ are defined as
    \begin{equation}
        \Psi_\al(t)=\frac12\|\partial_{t} v_n\|^2+\frac12\|\nabla v_{n}\|^2+\al \langle \partial_{t} v_n,v_{n}\rangle+\int_\Omega \(G(u)-G_n(u)\)\td x
    \end{equation}
    and
    \begin{equation}\label{eq5.39}
        N_n(t)=C\|g(u)-g_n(u)\|_{L^{\frac65}(\Omega)}+\al\int_\Omega \(G(u)-G_n(u)\)\td x+C\|\sigma_l(u)-\gamma\|_{L^{\frac32}(\Omega)}.
    \end{equation}
    Then by Gronwall lemma, we obtain
    \begin{equation}
        \Psi_\al(t) \le \Psi_\al(t_i)e^{-\al(t-t_i)}+\frac{C_4}{\al}\sup_{t_i\le s\le T_n}N_n(s).
    \end{equation}
    Recall that initial value $(v_n(t_i), \partial_{t} v_n(t_i))=(u(t_i)-\bar{u}_0, u_t(t_i)) \to (0, 0)$ in $\H$ as $i\to \infty$. Thus there exists a positive integer $M^\prime_n>M_n$ such that for all $i>M^\prime_n$,
    \begin{equation}
        \Psi_\al(t_i)\le \frac{C_4}{\al}\sup_{t_i\le s\le T_n}N_n(s).
    \end{equation}
    Thus, for $i>M^\prime_n$ and $t\ge t_i$,
    \begin{equation}
        \Psi_\al(t) \le \frac{2C_4}{\al}\sup_{t_i\le s\le T_n}N_n(s).
    \end{equation}

    On the other hand, combining with \eqref{wH2} and Agmon inequality $\|w_n\|_{\infty}^2\le c\|\nabla w_n\|\cdot\|\Delta w_n\|$, we derive
    \begin{equation}\label{}
        \|w_n\|_{\infty}\le  Cn\sup_{t_i\le s\le T_n}\|\nabla v_n(s)\|^{1/2}+C\le C_5 n\sup_{t_i\le s\le T_n}N^{\frac14}_n(s)+C_5.
    \end{equation}
    Let $A_n(t)=\{x\in \Omega : |u(t,x)|\ge n\}$. From the definition of $g_n$ and \eqref{F1}, we can obtain that
    \begin{equation}\label{eq5.44}
        C\|g(u(t))-g_n(u(t))\|_{L^{\frac65}(\Omega)}\le  C_6\(\int_{A_{n}(t)} |u(t,x)|^{6} \td x\)^{\frac56},
    \end{equation}
    and
    \begin{equation}\label{eq5.45}
        \int_\Omega \(G(u)-G_n(u)\)\td x \le  C_7\int_{A_{n}(t)} |u(t,x)|^{6} \td x.
    \end{equation}
    According to condition \eqref{S3}, $\sigma_l(u)=\gamma$ for $|u|\le l$, which yields
    \begin{equation}\label{eq5.46}
        \begin{aligned}
            C\|\sigma_l(u(t))-\gamma\|_{L^{\frac32}(\Omega)} & =C\(\int_{\Omega}|\sigma_l(u(t))-\gamma|^{\frac32}\td x\)^{\frac23}                         \\
                                                             & = C\(\int_{\{x\in \Omega :  |u(t,x)|>l\}}|\sigma_l(u(t))-\gamma|^{\frac32}\td x\)^{\frac23} \\
                                                             & \le C_8\(\int_{\{x\in \Omega :  |u(t,x)|>l\}} |u(t,x)|^{6}\td x\)^{\frac23}.
        \end{aligned}
    \end{equation}
    Under the condition $l\ge n$, the set inclusion $\{x\in \Omega : |u(t,x)|>l\}\subset A_{n}(t)$ implies
    \begin{equation}\label{eq5.47}
        \begin{aligned}
            C\|\sigma_l(u(t))-\gamma\|_{L^{\frac32}(\Omega)} \le C_8\(\int_{A_{n}(t)} |u(t,x)|^{6}\td x\)^{\frac23}.
        \end{aligned}
    \end{equation}
    By the compactness of $\mathscr{A}$, there exists $n_1 \in \mathbb{N}$ which is independent of $u(t)$, such that
    \begin{equation}\label{eq5.48}
        \sup_{t\in \R}\int_{A_{n_1}(t)} |u(t,x)|^6 \td x\le \(\frac{1}{3C_5(C_6+C_7+C_8)+1}\)^{4}.
    \end{equation}
    This guarantees
    $$C_5\sup_{t_i\le s\le T_n}N^{\frac14}_n(s)\le \frac13.$$
    Consequently, the uniform bound for $w_n$ follows
    \begin{equation}\label{5.52}
        \|w_n(t)\|_{\infty}\le  \frac n3+C_5,~~\forall t_i\le  t\le  T_{n_1}.
    \end{equation}
    Let $B_{\V}$ be the infum upper bound in the $\V$-norm of all equilibrium points of semigroup $S(t)$.
    Since
    $$|u(t)|=|w_n(t)+v_n(t)+\bar{u}_0|\le |v_n(t)|+\frac n3+C_5+ B_{\V},$$
    then for $n\ge n_1+6C_5+6B_{\V}$ and $x\in A_n(t)$, we have
    $$\begin{aligned}
            |v_n(t)| & \ge |u(t)|-\frac n3-C_5-B_{\V}\ge \frac {2n}3-C_5-B_{\V} \\
                     & \ge \frac {n}3+C_5+B_{\V}\ge|w_n(t)|+B_{\V}.
        \end{aligned}$$
    This yields the inclusion
    $A_n(t)\subset \hat{A}_n(t):=\{x\in \Omega : |v_n(t)|\ge |w_n(t)|+B_{\V}\}.$
    This inclusion permits the refined bound of $N_n(s)$,
    \begin{equation}\label{5.49}
        \begin{aligned}
            N_n(s)\le & (C_6+C_7+C_8)\int_{A_{n}(t)}|u(t)|^{6}\td x                                                                      \\
            \le       & C\left(\int_{A_n(t)} |u(t)|^6 \td x\right)^{\frac23}\left(\int_{\hat{A}_n(t)}|u(t)|^6\td x\right)^{\frac13}      \\
            \le       & C\left(\int_{A_{n}(t)}|u(t)|^{6}\td x\right)^{\frac23} \left(\int_{\hat{A}_n(t)}|v_n(t)|^6\td x\right)^{\frac13} \\
            \le       & C_9\left(\int_{A_{n}(t)}|u(t)|^{6}\td x\right)^{\frac23}\|\nabla v_n(t)\|^2, ~~\forall t_{i}\le  t\le  T_n.
        \end{aligned}
    \end{equation}
    By the compactness of $\mathscr{A}$ again, there exists  $n_0 \in \mathbb{N}$ which is greater than $n_1+6C_5+6B_{\V}$ such that
    \begin{equation}\label{eq5.51}
        \sup_{t\in \R}\int_{A_{n_0}(t)}|u(t)|^6\td x\le \left(\frac{\al}{8C_9}\right)^{\frac32},
    \end{equation}
    this implies
    \begin{equation}\label{eq5.52}
        N_n(s)\le \frac{\al}{8}\|\nabla v_n(t)\|^2, ~~\forall t_{i}\le  t\le  T_n.
    \end{equation}
    Substituting above inequality into \eqref{5.44}, we obtain
    \begin{equation}
        \begin{aligned}
            \frac{\td}{\td t}\Psi_\al(t)+\frac{\al}{2} \Psi_\al(t)\le0, ~~ t_{i}\le t\le T_{n_0}.
        \end{aligned}
    \end{equation}
    Then applying Gronwall lemma again, we have exponential decay estimate of energy
    \begin{equation}
        \Psi_\al(t) \le \Psi_\al(t_i)e^{-\frac{\al (t-t_{i})}{2}},
    \end{equation}
    holds on $[t_{i}, T_{n_0}]$.

    Finally, let $i\to \infty$, we can deduce \eqref{5.41}.
\end{proof}
By Lemmas \ref{lem5.3} and \ref{lem5.4}, for $l\ge n_0$, we have $(u(T_{n_0}), u_t(T_{n_0})) \in \V$ and
\begin{equation}
    \|(u(T_{n_0}), u_t(T_{n_0}))\|_{\V}\le r_0.
\end{equation}
Since $u(t)$ satisfies problem \eqref{eq1} on $(T_{n_0}, \infty) \times \Omega$ with initial value $(u(T_{n_0}), u_t(T_{n_0}))$. Applying Lemma \ref{unif_strogn_est}, we find that there exists a constant $l(r_0)>0$ such that $\|(u_0, u_1)\|_{\V} \le Q(r_0+\|\bar{u}_0\|_{H^2(\Omega)})$ for $l\ge \max\{n_0,l(r_0)\}$. Recall that all equilibrium points of $S(t)$ are uniformly bounded in $\V$, thus $\mathscr{A}$ is a bounded subset of $\V$. Furthermore, we can refine the bound of $\mathscr{A}$ by Lemma \ref{unif_strogn_est}.










\end{document}